\documentclass[a4wide,12pt,reqno]{amsart}
\usepackage{a4wide}
\usepackage[utf8]{inputenc} 
\usepackage{amsmath,amsthm}
\usepackage{amssymb}
\newtheorem{trm}{Theorem}[section]
\newtheorem{prop}[trm]{Proposition}

\newtheorem{lm}[trm]{Lemma}
\newtheorem{cor}[trm]{Corollary}

\theoremstyle{remark}
\newtheorem{rmk}{Remark}[section]
\newtheorem{rmk-pl}[rmk]{Remarks}

\newtheorem{ex-sg}[ex]{Example}
\theoremstyle{definition}
\newtheorem{dfn}{Definition}[section]

\def\R{\mathbb{R}}
\def\Li{\text{Li}}
\newcommand{\intn}[2]{\ensuremath{[\![ \, #1 \,;\, #2 \,]\!]}}
\newcommand{\intnn}[2]{\intn{#1}{#2}}

\def\gcd{\text{gcd}}
\def\1{\mathbb{1}}
\def\W{\overline{W}}
\def\Q{\mathcal{Q}}
\def\A{\mathcal{A}}
\def\N{\mathbb{N}}
\def\Z{\mathbb{Z}}

\def\Pr{\mathbb{P}}
\def\P{\mathcal{P}}

\def\F{\mathbb{F}}
\def\mod{\mathrm{\mbox{ }mod\mbox{ }}}
\def\Vol{\mathrm{Vol}}
\def\gcd{\mathrm{gcd}}
\newcommand{\abs}[1]{\left\lvert #1 \right\rvert}
\newcommand{\Zp}[2]{(\Z/p^{#1}\Z)^{#2}}
\newcommand{\nor}[1]{\left\lVert #1 \right\rVert}
\DeclareMathOperator{\E}{\mathbb{E}}

\DeclareMathOperator{\grad}{\overrightarrow{\mathrm{grad}}}

\title{Asymptotics for some polynomial patterns in the primes}
\author{Pierre-Yves Bienvenu}
\address{School of Mathematics, University of Bristol, Bristol BS8 1TW, United Kingdom}
\email{pb14917@bristol.ac.uk}
\date{\today}
\begin{document}
\maketitle
\begin{abstract}
We prove asymptotic formulae for
sums of the form
$$
\sum_{n\in\Z^d\cap K}\prod_{i=1}^tF_i(\psi_i(n)),
$$
where $K$ is a convex body, each $F_i$ is either the von Mangoldt function or the representation
function of a quadratic form, and $\Psi=(\psi_1,\ldots,\psi_t)$ is a system of linear forms of finite complexity. 
When all the functions $F_i$ are equal to the von Mangoldt function, we recover a result of Green and Tao, while when they are all representation functions of quadratic forms, we recover a result of Matthiesen. Our formulae imply asymptotics for some polynomial patterns in the primes.
Specifically, they describe the asymptotic behaviour of the number of $k$-term arithmetic progressions of primes whose common difference is a sum of two squares. 

The article combines ingredients from the work of Green and Tao on linear equations in primes and that of Matthiesen on linear correlations amongst integers represented by a quadratic form. To make the von Mangoldt function compatible with the representation function of a quadratic form, we provide a new pseudorandom majorant for both -- an average of the known majorants for each of the functions -- and prove that it has the required pseudorandomness properties.
\end{abstract}
\section{Introduction}
In a celebrated article \cite{GT1}, Green and Tao proved the following.
\begin{trm}
\label{GT1}
The set $\P$ of primes contains arbitrarily long arithmetic progressions.
% (abbreviated as APs). 
\end{trm}
This result has been strengthened and generalised in a number of 
ways over the last decade.
One crucial such strengthening was achieved by Green and Tao 
\cite{GT2}, when they proved the existence of prime solutions  to a large class of systems of linear equations. Moreover, they provided asymptotics for the number of such solutions. 
Among other things, they proved an asymptotic for the number of $k$-term arithmetic progressions
in the primes up to $N$.
Their result was conditional on two
conjectures that Green, Tao and Ziegler later proved completely
\cite{MN,GI}.
% one can even impose the common difference to be the double of a prime congruent to 1 modulo 4, thus a sum of two squares by a theorem of Fermat \cite[Chapter 1]{Cox}. But the vast majority of sums of two
%squares are not of this type, so Green and Tao's result is not very good at counting arithmetic progressions whose common difference is a sum of two squares. Our result is able to do this -- with a slight \emph{caveat}: the progressions are weighted, i.e. counted with multiplicity; see Corollary~\ref{cardinality3aps}.
%
%As often in number theory, the asymptotics of Green and Tao
%are more comfortably stated and proven with the
%\emph{von Mangoldt function} instead of the indicator function of the primes. Recall that the von Mangoldt function $\Lambda$ is
%defined on $\N$ by setting $\Lambda(n)=\log p$ if $n$ is a power of 
%a prime $p$ and $\Lambda(n)=0$ otherwise. The prime number theorem is the statement that this function has average up to $N$ converging to 1 as $N$ goes to infinity. 
%Roughly speaking, the method of Green and Tao is to show that this function is
%so \emph{uniform} that it can be replaced by the constant 
%function 1 in all sorts of reasonable averages.

Applying the same general method as Green and Tao, Matthiesen obtained
similar results with the divisor function \cite{Matt2} or
the representation function of quadratic forms \cite{Matt} instead of the von Mangoldt function.

%
%Another generalisation of Theorem \ref{GT1} was proved by
%Tao and Ziegler \cite{TaoZieg}. They determined lower bounds for the number of polynomial progressions in the primes, that is,
%patterns of the form $n+P_1(d),\ldots,n+P_k(d)$,
%where $P_1,\ldots,P_k$ are polynomials in $\Z[d]$ 
%vanishing at $0$. Since the submission of our paper,
%Tao and Ziegler \cite{polyprog} even obtained asymptotics for such patterns; their main theorem covers our Corollary \ref{cardinality3aps}, but their
%methods are more involved.

In this paper, we state
and prove a theorem (Theorem \ref{mytrm}) which encompasses both Green and Tao's and Matthiesen's results. It implies asymptotics for some -- admittedly very specific -- polynomial patterns. For instance, Corollary \ref{cardinality3aps} gives 
the asymptotic for the
number of $k$-term arithmetic progressions of primes whose 
common difference is a sum of two squares, where each such
progression is counted as many times as the common difference
is represented a sum of two squares. 

Since the first submission of 
our paper, Tao and Ziegler \cite{TZ} obtained asymptotics
for a much larger class of polynomial progressions, improving upon 
their earlier work \cite{TaoZieg} that gave only lower bounds. 
The methods they develop are vastly more intricate than ours, but there exist a number of interesting equations involving primes and sums of two squares which belong to the scope of our theorem but not to theirs.

\textbf{Acknowledgements}. The author thanks his supervisor Julia Wolf 
for useful conversations, guidance and reading many drafts, Sean Prendiville for suggesting  this problem, and Sam Chow and Andy Corbett for useful remarks.
\subsection{Preliminaries}
\label{sec:prelim}
We require a few definitions and some notation to be able to state
our theorem. The von Mangoldt function $\Lambda$ is
defined on $\N$ by setting $\Lambda(n)=\log p$ if $n$ is a power of 
a prime $p$ and $\Lambda(n)=0$ otherwise. We define for any integer $q$
the \emph{local von Mangoldt function} on $\Z$ by
$$\Lambda_q(n)=\frac{q}{\phi(q)}1_{(n,q)=1},$$
where $\phi$ is the Euler totient function defined by
$\phi(q)=\abs{(\Z/q\Z)^*}$ and $(n,q)$ is the greatest common divisor (gcd) of $n$ and $q$. In fact, $\Lambda_q$ can 
naturally be defined on
$\Z/kq\Z$ for any $k\in\N$.
Observe the use of the symbol
$1_P$ for a proposition $P$, which means 1 if $P$ is true and 0
otherwise.

Let $d,t\geq 1$ be integers.
An \emph{affine-linear} form $\psi$ on $\Z^d$ is a polynomial in $d$ variables of degree at most 1 with integer coefficients.
We denote by $\dot{\psi}$ its linear part; then $\psi=\dot{\psi}+\psi(0)$.

If $\psi_1,\ldots,\psi_{t}$ are affine-linear forms, we say that
$\Psi=(\psi_1,\ldots,\psi_{t}):\Z^d\rightarrow \Z^t$ is
a system of affine-linear forms.
It has \emph{finite complexity} if no two of the forms are affinely dependent, i.e. for any $i\neq j$, the
linear parts $\dot{\psi_i}$ and $\dot{\psi_j}$ are not proportional.

A \emph{binary quadratic form} is a polynomial
$$f(x,y)=ax^2+bxy+cy^2$$
where $a,b$ and $c$ are integers. Its \emph{discriminant} is $D=b^2-4ac$.
A \textit{positive definite} binary quadratic form (abbreviated as PDBQF) is a binary quadratic form of negative discriminant.
The \emph{representation function} of $f$ is the arithmetic function defined by
$$R_f(n)=\abs{\{(x,y)\in\Z^2\mid f(x,y)=n\}}.$$
For any integers $q$ and $\beta$, we let
$$\rho_{f,\beta}(q)=\abs{\{(x,y)\in [q]^2\mid f(x,y)\equiv\beta\mod q\}}.
$$

We shall use the notation
$$\E_{a\in A}=\frac{1}{\abs{A}}\sum_{a\in A}$$
to denote the averaging operator. We may also write
$\Pr_{a\in X}(a\in A)$ for $\abs{A}/\abs{X}$, for finite sets
$A\subset X$. The letter $p$ is reserved
for primes, the set of which is denoted by $\P$; for instance $\prod_p$ implicitly means $\prod_{p\in\P}$.
%For integers $a,b$, we write $\intnn{a}{b}=[a,b]\cap\Z$
%for an interval of integers and $[N]=\intnn{1}{N}$ for the
%set of the $N$ first integers.

The asymptotic parameter going to infinity is denoted by $N$. 
%The parameter $N'$ introduced later goes to 
%infinity with $N$ so that the limit when $N$ goes to
%infinity is the same as the limit when $N'$ goes to infinity.
We use the symbols $X\sim Y$ to say that $X/Y$
tends to 1 as $N$ tends to infinity. We shall use $X=O(Y)$
to say that $X/Y$ is bounded and $X=o(Y)$ to say that $X/Y$ 
tends to 0. Both $O$ and $o$ can be complemented with a subscript
indicating the dependence of the implied constant or the
implied decaying function. We also use $X\ll Y$,
which is synonymous to $X=O(Y)$ and can be complemented by subscripts as well.

\subsection{The main theorem}
We are now ready to state our main theorem.
\begin{trm}
Let $\Psi=(\psi_1,\ldots,\psi_{t+s}):\Z^d\rightarrow \Z^{t+s}$ be a system of affine-linear forms of finite complexity. 
Suppose that the coefficients of the
linear part $\dot{\Psi}$ are bounded\footnote{Green and Tao \cite{GT2} introduced the notion
of \emph{size at scale} $N$. One can check that the condition that the system has bounded size at scale $N$ is
equivalent to the boundedness of the linear part together with the condition on the image of $K$.} by some constant $L$.
Let $K\subset [-N,N]^d$ be a convex body such that $\Psi(K)\subset [0,N]^{t+s}$. Let $f_{t+1},
\ldots, f_{t+s}$ be PDBQFs of discriminants $D_j<0$ for $j=t+1,\ldots,t+s$.
Then
\begin{equation*} 
\sum_{n\in\Z^d\cap K}\prod_{i=1}^{t}\Lambda(\psi_i(n))\prod_{j=t+1}^{t+s}
R_{f_j}(\psi_j(n))=\beta_{\infty}\prod_{p}\beta_p+o(N^d),
%\label{mainasymptotic}
\end{equation*}
where
$$\beta_{\infty}=\Vol(K)\prod_{j=t+1}^{t+s}\frac{2\pi}{\sqrt{-D_j}}$$
and
$$\beta_p=\lim_{m\rightarrow +\infty}\E_{a\in (\Z/p^m\Z)^d}\prod_{i=1}^t\Lambda_p(\psi_i(a))\prod_{j=t+1}^{t+s}
\frac{\rho_{f_j,\psi_j(a)}(p^m)}{p^m}.$$
%with $\Lambda_p(n)=\frac{p}{\phi(p)}1_{(n,p)=1}$ and $\rho_{f,\beta}(q)=\abs{\{(x,y)\in [q]^2\mid f(x,y)\equiv\beta\mod q\}}$.
\label{mytrm}
\end{trm}
The error term is not effective (see \cite[Sections 13 and 14]{GT2} for a discussion) and the implied decaying 
function depends on $d,t,s,L$ and the discriminants.

Two important special cases arise when $s=0$ or $t=0$, that is,
when the functions featuring are either all equal to the von Mangoldt function, or all representation functions. Then one of the products is trivial.
%, hence equal to one.
\begin{itemize}
\item When $s=0$, one immediately recovers the result of Green and Tao \cite[Main Theorem]{GT2}.
Indeed, for $m\geq 1$, we have
$$
\E_{a\in (\Z/p^m\Z)^d}\prod_{i=1}^t\Lambda_p(\psi_i(a))=
\E_{a\in (\Z/p\Z)^d}\prod_{i=1}^t\Lambda_p(\psi_i(a))
$$
so that 
$$\beta_p=\E_{a\in (\Z/p\Z)^d}\prod_{i=1}^t\Lambda_p(\psi_i(a)).
$$
\item When $t=0$, Theorem \ref{mytrm} boils down to
the formula of Matthiesen \cite[Theorem 1.1]{Matt}.
\end{itemize}
For each prime $p$, we call $\beta_p$ the \textit{local factor} modulo $p$.
The existence of the limit
as $m$ tends to infinity that defines it is proven
in Proposition \ref{prop:betapm}; the convergence of the infinite
product $\prod_p \beta_p$ is a consequence of \ref{lm:locfac2}.

Sometimes one can get an asymptotic even when the system has infinite complexity, but the asymptotic takes a completely different form then. For instance it is easy to see that
$$
\sum_{n\leq N}\Lambda (n) R(n) \sim 8\sum_{\substack{p\leq N\\p\equiv 1\mod 4}}\log p\sim 4N
$$
by Fermat's theorem on sums of two squares and the prime number theorem in arithmetic progressions.
We do not address such systems in this paper.

\subsection{Progressions of step a sum of two squares in the primes}
Here $R$ and $\rho$ (see Section \ref{sec:prelim}) will implicitly refer to the form
$f(x,y)=x^2+y^2$
whose discriminant is $-4$.
Our application concerns arithmetic progressions in the primes whose common difference is required to be a sum of two squares.
It shows that the Green-Tao theorem (case $s=0$ of Theorem \ref{mytrm}) holds not only
for linear systems, but also for some -- admittedly very specific --
polynomial systems.
\begin{cor}
\label{cardinality3aps}
Let $k\geq 1$ be an integer and $$L=\{(a,b,c)\in\R^3\mid 1\leq a\leq a +(k-1)(b^2+c^2)\leq N\}.$$
Let $\Psi=(\psi_0,\cdots,\psi_{k-1})\in\Z[a,b,c]^k$ be the polynomial system defined by
$$
\psi_i(a,b,c)=a+i(b^2+c^2).
$$
Then
\begin{equation}
\label{polysyst}
\sum_{n\in \Z^3\cap L}\prod_{i=0}^{k-1}\Lambda(\psi_i(n))
=\beta_{\infty}\prod_p\beta_p+o(N^2)
\end{equation}
with $\beta_{\infty}=\Vol(L)$ and
\begin{equation}
\label{localfacAP}
\beta_p=\E_{n\in\Zp{}{3}}\prod_{i=0}^{k-1}\Lambda_p(\psi_i(n)).
\end{equation}
\end{cor}
As noted in the introduction,
Corollary \ref{cardinality3aps} now appears as a special case of
a very recent
result of Tao and Ziegler \cite[Theorem 1.4]{TZ}. However, our method can deal with the variant where $L$ is replaced by $[N]\times [\sqrt{N}\log^{-A}N]^2$ for any constant $A>0$, thus the common difference of the progression is markedly
smaller than the terms of the progression. 
Indeed, in the proof below, we reduce equation \eqref{polysyst}
to one involving a linear system, to which we apply \cite[Theorem 1.3]{TZ}. When one considers the system as a polynomial one, one cannot restrict $b$ and $c$ to such a small range.
\begin{proof}[Proof of Corollary~\ref{cardinality3aps} assuming Theorem \ref{mytrm}]
We note that the left-hand side of equation~\eqref{polysyst} can be
written as
\begin{equation}
\label{kap}
\sum_{(a,d)\in\Z^2\cap K}\Lambda(a)\Lambda(a+d)\cdots\Lambda(a+(k-1)d)R(d),
\end{equation}
where $K=\{(a,d)\in\R^2\mid 1\leq a\leq a+(k-1)d\leq N\}$
is a convex body in $\R^2$.
Applying Theorem~\ref{mytrm} to this convex body and
the system
$(a,d)\mapsto (a,a+d,\ldots,a+(k-1)d,d)$, which is of finite
complexity, we get
\begin{equation} 
\sum_{n\in \Z^3\cap L}\prod_{i=0}^{k-1}\Lambda(\psi_i(n))=\beta_{\infty}\prod_p\beta_p+o(N^2)
\label{mainasymptoticap}
\end{equation}
with
$\beta_{\infty}=\pi \frac{N^2}{2(k-1)}$ and
$$\beta_p=\lim_{m\to\infty}\E_{(a,d)\in (\Z/p^m\Z)^2}\frac{\rho_d(p^m)}{p^m}\left(\frac{p}{\phi(p)}
\right)^k
\prod_{i=0}^{k-1}1_{(a+id,p)=1}.$$
It is easy to see that $\Vol(L)=\beta_{\infty}$.
It remains to prove that the local factors have the form \eqref{localfacAP}. First,
\begin{equation}
\label{liftinvar}
\begin{split}
\E_{(a,d)\in (\Z/p^m\Z)^2}\frac{\rho_d(p^m)}{p^m}
\prod_{i=0}^{k-1}1_{(a+id,p)=1} &=\E_{(a,b,c)\in (\Z/p^m\Z)^3}\prod_{i=0}^{k-1}1_{(a+i(b^2+c^2),p)=1}.
%\\
%&=\frac{1}{p^m}\sum_{a\in\Zp{m}{*}}(1-\sum_{i=0}^{k-1}\Pr_{(b,c)\in \Zp{m}{2}}(b^2+c^2\equiv -\overline{i}a\mod p))
\end{split}
\end{equation}
Now let $a\mapsto\tilde{a}$ be the canonical map $\Z/p^m\Z\rightarrow\Z/p\Z$. We notice that it is a
$p^{m-1}$-to-1 map and that $(a+i(b^2+c^2),p)=1$
if and only if $(\tilde{a}+i(\tilde{b}^2+\tilde{c}^2),p)=1$.
Hence
$$
\E_{(a,b,c)\in (\Z/p^m\Z)^3}\prod_{i=0}^{k-1}1_{(a+i(b^2+c^2),p)=1}=
\E_{(a,b,c)\in (\Z/p\Z)^3}\prod_{i=0}^{k-1}1_{(a+i(b^2+c^2),p)=1}
$$
does not depend on $m$ and the local factors are of the
desired form.
\end{proof}
Let us compute explicitly the local factors $\beta_p$.
Suppose first that $p\geq k$.
We remark that
$$\beta_p=\left(\frac{p}{p-1}\right)^k
\frac{1}{p}\sum_{a\in\Zp{}{*}}(1-\sum_{i=1}^{k-1}\Pr_{(b,c)\in \Zp{}{2}}(b^2+c^2\equiv -\overline{i}a\mod p)),
$$
where $\overline{i}$ is the inverse of $i$ modulo $p$.
%Indeed the canonical map $\Zp{m}{2}\rightarrow\Zp{}{2}$ induces
%a $p^{2m-2}$-to-1 map between the representations modulo $p^m$ and the representation modulo $p$. 
%Thus we can rewrite the last term of \eqref{liftinvar}
%as
%$$
%\frac{1}{p}\sum_{a\in\Zp{}{*}}(1-\sum_{i=0}^{k-1}\Pr_{(b,c)\in \Zp{}{2}}(b^2+c^2\equiv -\overline{i}a\mod p))
%$$
%which does not depend on $m$ anymore.
%This implies that
%$$
%\beta_p=\left(\frac{p}{p-1}\right)^k\Pr_{(a,b,c)\in\Zp{}{3}}
%\left(\left(\prod_{i=0}^{k-1}(a+i(b^2+c^2)),p\right)=1\right).
%$$
Moreover, for any $a\in\Zp{}{*}$, setting 
$e(x)=\exp(2i\pi x)$ as customary,
we have 

\begin{align*}
\abs{\{(b,c)\in \Zp{}{2}\mid b^2+c^2\equiv a \mod p\}}&=
\sum_{(b,c)\in \Zp{}{2}}\frac{1}{p}
\sum_{h\in \Z/p\Z}e\left(\frac{h(b^2+c^2-a)}{p}\right)\\
&=\frac{1}{p}\left(\sum_{h\in\Zp{}{*}}e\left(-\frac{ha}{p}\right)
\left(\sum_{b\in\Z/p\Z}e\left(\frac{hb^2}{p}\right)\right)^2+p^2\right)\\
&= \left\lbrace
\begin{tabular}{ll}
$p-1$ & \text{ if } $p\equiv 1\mod 4$\\ 
$p+1$ & \text{ if } $p\equiv -1\mod 4$\\
$p$ &\text{ if } $p=2$.
\end{tabular} \right.
\end{align*}

The last equality follows from the classical computation of
Gauss sums
(see \cite[3.38]{IK}).
For $p\geq k$, this leads to
$$
\beta_p=\left\lbrace\begin{array}{ll}
\left(1+\frac{1}{p-1}\right)^k\left(1-\frac{k}{p}+2\frac{k-1}{p^2}-\frac{k-1}{p^3}\right) & \text{ if } p\equiv 1\mod 4\\ 
\left(1+\frac{1}{p-1}\right)^k\left(1-\frac{k}{p}+\frac{k-1}{p^3}\right) & \text{ if } p\equiv -1\mod 4.\\
\end{array} \right.
$$

It is easy to compute the local factors for $p\leq k$. We find that

$$\beta_p=\left\lbrace
\begin{array}{ll}
\left(\frac{p}{p-1}\right)^k\frac{(p-1)(2p-1)}{p^3} & \text{ if } p\equiv 1\mod 4\\ 
\left(\frac{p}{p-1}\right)^k\frac{p-1}{p^3} & \text{ if } p\equiv -1\mod 4\\
2^{k-2} &\text{ if } p=2.
\end{array} \right.$$
We notice that $\beta_p$ is nonzero for every $p$ and that
$\beta_p=1+O(p^{-2})$, thus $\prod_p\beta_p$ is a nonzero convergent product. We prove in Lemma \ref{lm:locfac2} that the product of the local factors is always convergent for systems of finite
complexity.

%Their theorem implies that
%$$
%\sum_{\substack{n,d\\n+(k-1)d^2\leq N}}\prod_{i=0}^{k-1}\Lambda(n+id^2)
%\gg_k N^{3/2}.
%$$
%An asymptotic equality
%instead of a lower bound remains out of reach in this case, and our method does not apply.

%Arithmetic progressions whose common difference is a square
%are unfortunately truly polynomial patterns.

Corollary \ref{cardinality3aps} counts the number of \emph{weighted} arithmetic progressions of primes up to $N$ whose common difference is a sum of two squares, each such arithmetic progression being weighted by the number of representation of the common difference.
To count these progressions without multiplicity, one has to replace $R$ by the indicator function
$1_S$ of the sums of two squares. The very recent work of Matthiesen on multiplicative functions
\cite{multip}
shows how to deal with $1_S$ in linear averages, so that the
count without multiplicity can be derived along the same lines as the count with multiplicity. We
refrain from doing it here for brevity, but we note that this differs from the results of Tao and Ziegler \cite{TZ},
which count necessarily multiplicities.

In general, the only polynomial patterns we are able
to deal with are the ones which can be converted
into linear patterns by the use of representation functions
of PDBQFs, as in the proof of Corollary \ref{cardinality3aps}. The ability to deal with arithmetic progressions whose common difference is a sum of two squares as if they were
a linear pattern
is reminiscent of a result of Green \cite{Gr}: he proved that if a
set $A\subset [N]$ does not contain any such progression of length 3, then $\abs{A}\ll N(\log\log N)^{-c}$ for some $c>0$.
%This is a bound of the same shape as the one for sets without any progression of length 4 at all.
%In fact, Green managed to control the pattern $a,a+x^2+y^2,a+2(x^2+y^2)$ by Gowers' $U^3$ norm, just as
%progressions of length 4.

\subsection{Tuples of primes whose pairwise midpoints are sums of
two squares}
Theorem \ref{mytrm} can yield many further asymptotics for the
number of solutions to equations in primes and sums of squares, some of which are not covered by Tao and Ziegler \cite{TZ}.
This is the case of the equation $p_1+p_2=n_1^2+n_2^2$ with $p_1,p_2$ primes and $n_1,n_2$ integers. The underlying polynomial
system is
$$
(n_1,n_2,n_3)\mapsto (n_1,n_2^2+n_3^2-n_1),
$$
which is not of the form $(n_1,n_1+P(\mathbf{r}))$, hence not a polynomial progression.
More generally,
by analogy to a theorem of Balog \cite{Balog}, we consider
tuples of odd primes $p_1,\ldots,p_d$ such that $(p_i+p_j)/2$ is
a sum of two squares for all $i\neq j$.
To determine the asymptotic for such tuples, we analyse the sum
$$
\sum_{1\leq n_1,\ldots,n_d\leq N}\prod_{i\in[d]}\Lambda(2n_i+1)
\prod_{j\neq k}R(n_j+n_k+1)
$$
where $R$ is again the representation function of sums of two squares. The system of linear forms at hand is of finite complexity,
so that Theorem \ref{mytrm} applies.

\subsection{Other results within the scope of our method}
We claim, but we do not formally prove, that our method yields
a result similar to Theorem \ref{mytrm}
with the divisor function $\tau$ instead of the
representation functions $R_{f_i}$. In fact, this result is
easier to prove, since the treatment of the
representation function of
a binary quadratic form by Matthiesen \cite{Matt} relies
on her earlier paper on the divisor function \cite{Matt2}.
%So the aforementioned theorem,
%inspired from \cite{GT2} and
%\cite{Matt2}, would be the following.
\begin{trm}
\label{trm:lambdadiv}
Let $\Psi=(\psi_1,\ldots,\psi_{t+s}):\Z^d\rightarrow \Z^{t+s}$ be a system of affine-linear forms of finite complexity. 
Suppose that the coefficients of the
linear part $\dot{\Psi}$ are bounded by $L$.
Let $K\subset [-N,N]^d$ be a convex body such that $\Psi(K)\subset [0,N]^{t+s}$. Write
$\Phi=(\psi_{t+1},\ldots,\psi_{t+s})$ and $\dot{\Phi}$ for the
linear part.
%Let $f_{t+1},
%\ldots, f_{t+s}$ be positive definite quadratic forms, of discriminants $D_i<0$ for $i=t+1,\ldots,t+s$.
Then
\begin{equation*} 
\sum_{n\in\Z^d\cap K}\prod_{i=1}^{t}\Lambda(\psi_i(n))\prod_{j=t+1}^{t+s}
\tau(\psi_j(n))=(\log N)^s\beta_{\infty}\prod_{p}\beta_p+o_{d,t,s,L}(N^d\log^sN)
%\label{mainasymptotic}
\end{equation*}
where
$$\beta_{\infty}=\Vol(K)$$
and
$$\beta_p=
\left(\frac{p}{p-1}\right)^{t-s}\E_{a\in [p]^d}
\prod_{i=1}^t1_{(\psi_i(a),p)=1}\sum_{(k_1,\ldots,k_s)\in\N^s}
\alpha_{\Phi_{a,p}}(p^{k_1},\ldots,p^{k_s})$$
%\E_{a\in (\Z/p^m\Z)^d}\prod_{i=1}^t\Lambda_p(\psi_i(a))\prod_{j=t+1}^{t+s}
%\frac{\rho_{f_j,\psi_j(a)}(p^m)}{p^m}
with $\Phi_{a,p}:b\mapsto \Phi(a)+p\dot{\Phi}(b)$  and $\alpha$ as in Definition \ref{def:locdens}.
\end{trm}

This theorem provides an asymptotic for the number of triples
of
nonnegative integers $(a,b,c)$ such that $a,a+bc,a+2bc$ are primes.
This is again a quadratic pattern; in fact, $\tau$ can be viewed as the representation function of the quadratic form
$(x,y)\mapsto xy$. We can obtain a result similar to
 Corollary \ref{cardinality3aps}.
We let $$L=\{(a,b,c)\in[1,+\infty[^3\mid a+(k-1)bc\leq N\}.$$ This is
not a convex body, but we have $\Vol(L)\sim \abs{L\cap\Z^3}\sim
N^2\log N/(k-1)$. 
It is not difficult to deduce from Theorem \ref{trm:lambdadiv}
that
\begin{equation*}
\sum_{(a,b,c)\in L\cap\Z^3}\prod_{i=0}^{t-1}\Lambda(a+ibc)
=\Vol (L)\prod_p\beta_p +o(N^2\log N)
\end{equation*}
with
$$
\beta_p=\prod_{i=0}^{t-1}\Lambda_p(a+ibc).
$$
Again this result has the same shape as the Green-Tao theorem
although the configuration involved is nonlinear.

We remark that the idea of mixing $\Lambda$ and $\tau$ is quite old. Titchmarsh \cite{Tit} considered sums
such as
$$\sum_{p\leq N}\tau(p+a)$$
or equivalently
$$
\sum_{n\leq N}\Lambda(n)\tau(n+a)
$$
for $a\in\Z$.
Assuming the Riemann hypothesis, he proved that that
$$
\sum_{n\leq N}\Lambda(n)\tau(n+a)=c_1(a)x\log x + O(x\log\log x)
$$
for some explicit constant $c_1(a)$. The result was proven unconditionally by Linnik \cite{Lin}. Fouvry \cite{Fouv}
proved the refined asymptotic formula
$$
\sum_{n\leq N}\Lambda(n)\tau(n+a)=c_1(a)x\log x+C_2(a)\Li(x)+O_A(x(\log x)^{-A})
$$
for any $A>0$.
Notice that this problem does not belong to the scope of our method, because the involved linear
system is of infinite complexity.

We also mention that Matthiesen, together with Browning \cite{MattBr}, was able to generalise her result about quadratic forms
to norm forms originating from a number field. This implies a generalisation of Theorem \ref{mytrm}, but we
refrain, for the sake of simplicity, from inspecting this general case.
\subsection{Overview of the general strategy}
We now turn to a proof of the main theorem, Theorem \ref{mytrm}.
The proof follows the usual Green-Tao method.
In Section 2, we perform the $W$-trick to suppress the preference of the von Mangoldt function
and the representation function for some residue classes.
Because of the notably different behaviours of these functions
with respect to arithmetic progressions, this is a delicate matter. Assuming some convergence properties of the local factors, which we prove in Appendix A, the implementation of the $W$-trick reduces the main theorem to Theorem~\ref{reduction}, the statement that
a multilinear average
$$
\E_{n\in\Z^d\cap K}(F_0(\psi_0(n))-1)\prod_{i=1}^tF_i(\psi_i(t))
$$
is asymptotically $o(1)$.
Thanks to a \emph{generalised von Neumann theorem},
it suffices to ensure that $F_0-1$ has small
\emph{Gowers uniformity norm}
and that all the functions $F_i$ and $F_0-1$ are bounded by a common \emph{enveloping sieve} or \emph{pseudorandom
majorant}. This is where the novelty of our paper lies. 
While individual pseudorandom majorants for $\Lambda$
and for $R_f$ are known, we need to construct a common one that works for $\Lambda$ and $R_f$ simultaneously.

We state the von Neumann theorem and prove Theorem~\ref{reduction} in Section 3, assuming the majorant
introduced then is sufficiently pseudorandom.
The required pseudorandomness property is proven in Appendix B. Appendix C provides some
general background around the notion of local density, i.e. the density of zeros of a linear system modulo a prime power.

%The possibility to combine $\Lambda$ with other suitably
%uniform arithmetic functions could open the door to
%the determination of the asymptotic frequency of various patterns in the future.

\section{Proof of Theorem \ref{mytrm}}
We fix some arbitrarily large integer $N$, so that
our asymptotic results are valid in the limit where $N$ tends to 
infinity.
We use the notation $[N]$ for the set of the first $N$ integers. 
Many of the parameters introduced in the sequel
implicitly depend on $N$ (such as the convex body $K$,
the map $p\mapsto\iota(p)$, the numbers $w,W,\overline{W}$, the set $X_0$...). 
%We will use the notation $[N]=\intnn{1}{N}$
%where $\intnn{a}{b}=\{a,a+1,\ldots,b\}$ for integers $a\leq b$.
\subsection{Elimination of a negligible set}
We start our proof by taking care of a
technicality. 
We would like to eliminate slightly awkward integers from the support of the von Mangoldt and the representation functions.
In fact, it will turn out handy to exclude prime powers and small primes from the support of $\Lambda$, so we introduce
$\Lambda'=1_{\mathcal{P}\setminus [N^{2\gamma}]}\log$, for some
constant
$\gamma\in (0,1/2)$ to be fixed later. It coincides with $\Lambda$ on the bulk of its support up to $N$, namely
large primes.

Similarly,
there is a fairly sparse subset $X_0\subset [N]$, depending on some constants $C_1>0$ and $\gamma>0$, on which the divisor function, and also the representation function, behave 
abnormally, so that our process of majorising by a
pseudorandom measure (carried out in Section 4) fails there. We recall
the following definition originating from \cite{Matt2}
and taken up in \cite{Matt}.
\begin{dfn}
Let $\gamma=2^{-k}$ for some $k\in\N$ and let $C_1>1$.
We define $X_0=X_0(\gamma,C_1,N)$ to be the set containing 0
and the set of positive integers $n\leq N$ satisfying either 
\begin{enumerate}
\item $n$ is excessively ``rough", i.e. divisible by some large prime power $p^a >\log^{C_1}N$ with $a\geq 2$, or
\item $n$ is excessively smooth in the sense that if $n=\prod_pp^{a_p}$ then
$$
\prod_{p\leq N^{(1/\log\log N)^3}}p^{a_p}\geq N^{\gamma/\log\log N}
$$
or
\item $n$ has a large square divisor $m^2\mid n$, which satisfies $m>N^{\gamma}$.
\end{enumerate}
\label{X0}
\end{dfn}
The following lemma,
which is Lemma 3.2 from \cite{Matt}, itself a synthesis
of Lemmas 3.2 and 3.3 from \cite{Matt2}, shows how negligible this set is.
\begin{lm}\label{skinny}
For $\Psi$ and $K$ as in Theorem \ref{mytrm}, we have
$$
\E_{n\in K\cap\Z^d}\sum_{i=t+1}^{t+s}1_{\psi_i(n)\in X_0}
\ll_{\gamma,d,s}\log^{-C_1/2}N.
$$
\end{lm}

This enables us to state the next lemma, which allows us to ignore $X_0$ altogether. For any PDBQF $f$, we use to the notation $\overline{R_f}(n)$ to denote $1_{n\notin X_0}R_f(n)$.

\begin{lm}
If the parameter $C_1$ in Definition \ref{X0} is large enough, and for any choice of the constant $\gamma\in (0,1/2)$,
Theorem \ref{mytrm} holds if and only if, under the same conditions, we have
\begin{equation}
\label{eq:exceptional}
\sum_{n\in K\cap\Z^d} \prod_{i=1}^t\Lambda'(\psi_i(n))\prod_{j=t+1}^{t+s}
\overline{R}_{f_j}(\psi_j(n))=\beta_{\infty}\prod_p\beta_p+o(N^d).
\end{equation}
\label{exceptional}
\end{lm}
\begin{proof}
We show first that
$$
\sum_{\substack{n\in K\cap\Z^d\\\exists j\in\intnn{t+1}{t+s}
\colon
\psi_j(n)\in X_0}} \prod_{i=1}^t\Lambda(\psi_i(n))\prod_{j=t+1}^{t+s}
R_{f_j}(\psi_j(n))=o(N^d).
$$
Notice the use of the notation
$\intnn{t+1}{t+s}=\{t+1,\ldots,t+s\}$.
We get rid of the von Mangoldt factors by bounding their product by $\log^t N$.
Then we use the Cauchy-Schwarz inequality followed by
the triangle inequality, which implies that
$$\left(\sum_{\substack{n\in K\cap\Z^d\\\exists j\in\intnn{t+1}{t+s}\colon
\psi_j(n)\in X_0}} \prod_{j=t+1}^{t+s}
R_{f_j}(\psi_j(n))\right)^2
\leq \sum_{n\in K\cap\Z^d}\left(\prod_{j=t+1}^{t+s}
R_{f_j}(\psi_j(n))\right)^2\sum_{n\in K\cap\Z^d}\sum_{j=t+1}^{t+s}1_{\psi_j(n)\in X_0}.$$
Finally, we use Lemma 3.1 of \cite{Matt}
which ensures that the first factor is $N^d\log^{O_s(1)}N$
while the second is $N^{d}\log^{-C_1/2}N$ according to
Lemma~\ref{skinny}, so that taking 
$C_1$ larger than $2(t+O_s(1))$, we have the result.

To replace $\Lambda$ by $\Lambda'$, we remark that for each
$i\in [t]$, the number of $n\in K\cap \Z^d$ such that $\psi_i(n)\leq N^{2\gamma}$, resp. $\psi_i(n)$ is a prime power
and not a prime, is $O(N^{d-1+2\gamma})$, resp. $O(N^{d-1}\log N\sqrt{N})$. Using Cauchy-Schwarz or even pointwise bounds such as the divisor bound $R_{f_j}(n)\ll \tau (n)\ll_{\epsilon} N^{\epsilon}$,
we conclude the proof of Lemma 2.2.
\end{proof}
%\begin{rmk}
From now on,
we will drop the bar, so that $R_f$ coincides with the actual
representation function of $f$ on $[N]\setminus X_0$
and is 0 on $X_0$.

%\label{warning}
%\end{rmk}
\subsection{Implementation of the $W$-trick}
The $W$-trick is by now fairly standard; see \cite{GT1}, \cite{GT2} for its implementations by Green and Tao, see also \cite{Matt2} and \cite{Matt}, which we are going to follow more closely.
The idea is to eliminate the obvious bias of the primes, like the strong preference for odd numbers, to produce a more uniform set. The representation function
of a PDBQF is also biased (it does not have the same average on every residue class),
so this has to be corrected, too. To do this we introduce, for some slowly growing function of $N$, such as 
$w(N)=\log\log\log N$, the products
$$W=\prod_{p\leq w}p\qquad\text{ and } \qquad\overline{W}=\prod_{p\leq w}p^{\iota(p)},$$
where $\iota (p)$ is defined
by
\begin{equation}
\label{def:iota}
p^{\iota(p)-1}<\log^{C_1+1}N\leq p^{\iota(p)}
\end{equation}
for some $C_1$ large enough as in Lemma \ref{exceptional}.
We observe that 
$$
\overline{W}\leq \prod_{p\leq w}p\log^{C_1+1}N\ll \exp ((C_1+1)w\log\log N))
$$
which is less than any power of $N$. In particular, we can decide that $\overline{W}<N^{\gamma}-1$ by choosing $N$ large enough.

Green and Tao did not need prime powers in their $W$, but in the case of a 
representation function of a PDBQF, they turn out to be necessary.
%This is of course linked to the Remark~\ref{unfortunately}. 
Notice that for $N$ large enough, for $p\leq w(N)= \log\log\log N$, we always have $\iota(p)\geq 1$.
We also introduce, for $b\in [\overline{W}]$, the function $\Lambda'_{b,\overline{W}}$ defined by
\begin{equation}
\label{Wtrickedvm}
\Lambda'_{b,\overline{W}}(n)=\frac{\phi(W)}{W}\Lambda' (\overline{W}n+b),
\end{equation}
where $\Lambda'=1_{\P}\log$
(we recall that $\phi(W)/W=\phi(\overline{W})/\overline{W}$).

Unfortunately, $$\overline{W}\geq\exp(w(N)C_1\log\log N)\gg_A \log^A N$$ for any $A>0$, so it will not be possible, without
a notable strengthening of the Siegel-Walfisz theorem, to
claim that $\Lambda'_{b,\overline{W}}$ has average $1+o(1)$.
A fortiori, it will not be possible to claim that
$\Lambda'_{b,\overline{W}}-1$ has the required uniformity property, in contrast to the normal $W$-trick.
We will be able to make do without this uniformity result.
%The advantage of this function is that it has average 1, not only in $[N]$ but
%in any arithmetic progression of $w(N)$-smooth modulus. Indeed, let $q_1$ be an integer whose prime factors are
%all at most $w(N)$ (what we call $w(N)$-\textit{smooth} or \textit{friable}), and
%let $q_0\in [q_1]$ and $P=\{q_0+mq_1\mid m\leq M\}$ and $b\in[\overline{W}]$ coprime to $W$.
%Then
%$$
%\E_{n\in P}\Lambda'_{b,\overline{W}}(n)=\frac{1}{M}\sum_{m\leq M}
%\frac{\phi(W)}{W}
%\Lambda'(\overline{W}q_1m+\overline{W}q_0+b)\sim \frac{\phi(W)}{W}\frac{\overline{W}q_1}{\phi(\overline{W}q_1)}=1
%$$
%by the prime number theorem in arithmetic progressions and the fact that $\phi(q)/q$
%only depends on the radical of $q$.

We perform the $W$-trick on $R_f$ as well, for any PDBQF $f$ of discriminant $D$.
Following Matthiesen \cite[Definition 7.2]{Matt}, we define 
\begin{equation}
\label{Wtrickedrep}
r'_{f,b}(m):=\frac{\sqrt{-D}}{2\pi}
\frac{\overline{W}}{\rho_{f,b}(\overline{W})}R_f(\overline{W}m+b),
\end{equation}
for any $b$ such that
%$b\neq 0\mod p^{\iota (p)}$ and 
$\rho_{f,b}(\overline{W})>0$.\footnote{Even if $\rho_{f,b}(\overline{W})=0$, we will use the notation
$r'_{f,b}(m)$ with the understanding that it means 0.} By construction, $R_f(n)$ equals 0
in the case where $n\in X_0$, in particular in the case
where $n\equiv 0\mod p^{\iota(p)}$ with $p\leq w(N)$. Hence, $r'_{f,b}=0$
if $b\equiv 0\mod p^{\iota(p)}$, so $r'_{f,b}$ has the desired property. In fact (see \cite[Definition 7.2]{Matt})
for $b\neq 0\mod p^{\iota(p)}$ and any $p\leq w(N)$ satisfying
$\rho_{f,b}(\overline{W})=0$, we have
$$
\E_{n\leq M}r'_{f,b}(n)=1+O(\overline{W}^3M^{-1/2}).
$$ 
This average in arithmetic progressions relies on elementary convex geometry and is valid uniformly in the modulus, in sharp
contrast with the analogous result for primes.

We now decompose the left-hand side of \eqref{eq:exceptional} into sums over congruence classes. We write
$$
\Z^d\cap K =\bigcup_{a\in [\overline{W}]^d} \Z^d\cap K_a ,
$$
where $$K_a=\{x\in \R^d\mid \overline{W}x+a\in K\}$$
is again a convex body.
Putting $$F(n)=\prod_{i=1}^t\Lambda(\psi_i(n))\prod_{j=t+1}^{t+s}
R_{f_j}(\psi_j(n)),$$
we can write the left-hand side of \eqref{eq:exceptional} as
\begin{equation}
\label{sumovercongclasses}
\sum_{n\in \Z^d\cap K}F(n)=\sum_{a\in [\overline{W}]^d}\sum_{n\in 
\Z^d\cap K_a}F(\overline{W}n+a).
\end{equation}
Moreover, for $j\in[t+s]$, we can write $\psi_j(\W n+a)=\W\tilde{\psi_j}(n)+c_j(a)$ 
where $c_j(a)\in [\overline{W}]$ and $\tilde{\psi_j}$ 
is an affine-linear form
differing from $\psi_j$ only in the constant term.
We remark that if $\psi_i(a)$ is not coprime to $W$ for $i\in [t]$
or if $\rho_{f_j,\psi_j(a)}(\overline{W})=0$ or $\psi_j(a)\equiv 0\mod p^{\iota(p)}$ for
some $j\in \intnn{t+1}{t+s}$ and some prime $p\leq w(N)$, then for each $n\in K_a\cap\Z^d$ we have
$F(n)=0$\footnote{Even if $(\psi_i(a),W)>1$, the integer
$\psi_i(a)$ could still be a prime $p\leq w(N)< N^{\gamma}$, but given that primes smaller than $N^{\gamma}$ are not in the support of $\Lambda'$, we still have $F(n)=0$.}. Thus the residues $a$  which 
bring a nonzero contribution to the right-hand side of
\eqref{sumovercongclasses}
are all mapped by
$\Psi$ to tuples $(b_1,\ldots,b_{t+s})$ belonging to the
following set.
\begin{dfn}
\label{bts}
We denote by $B_{t,s}$ the set of residues $b\in [\overline{W}]^{t+s}$
such that
\begin{enumerate}
\item for any $i\in[t]$, $(b_i, W)=1$;
\item for any $ j\in\intnn{t+1}{t+s}$
and any prime $p\leq w(N)$, we have $b_j\neq 0\mod p^{\iota(p)}$;
\item for any $ j\in\intnn{t+1}{t+s}$, $b_j$ is representable
by $f_j$ modulo $\overline{W}$, that is, 
$\rho_{f_j,b_j}(\overline{W})>0$.
\end{enumerate}
Moreover, for an affine-linear system $\Psi : \Z^d\rightarrow \Z^{t+s}$, we define $A_{\Psi}$ 
to be the set of all $a\in [W]^d$ such that
$(c_i(a))_{i\in[t+s]}\in B_{t,s}$. We recall that $c_i(a)$ is the reduction modulo $\overline{W}$ in $[\overline{W}]$ of $\psi_i(a)$; we will also denote by $c(a)$ the vector
$(c_i(a))_{i\in[t+s]}$. 
We usually drop the subscripts on $B_{t,s}$ and $A_{\Psi}$ when no ambiguity is possible.
\end{dfn}
Now we rewrite \eqref{sumovercongclasses} as
\begin{equation}
\begin{split}
\sum_{n\in \Z^d\cap K}F(n) &=\sum_{a\in A_{\Psi}}
\left(\frac{W}{\phi(W)}\right)^t\prod_{j=t+1}^{t+s}
\frac{2\pi}{\sqrt{-D_j}}\frac{\rho_{f_j,\psi_j(a_j)}(\overline{W})}{\overline{W}}\sum_{n\in K_a\cap\Z^d} F'_a(n)\\
&=\prod_{j=t+1}^{t+s}\frac{2\pi}{\sqrt{-D_j}}\sum_{a\in[\overline{W}]^d}Q(a)\sum_{n\in K_a\cap\Z^d} F'_a(n),\end{split}
\end{equation}
where
$$
Q(a)=\prod_{i=1}^t\Lambda_{W}(\psi_i(a))
\prod_{j=t+1}^{t+s}\frac{\rho_{f_j,\psi_j(a)}(\overline{W})}{\overline{W}}1_{\forall p\leq w,\psi_j(a)\neq 0\mod p^{\iota(p)}},
$$
and
$$
F'_a(n)=\prod_{i=1}^t\Lambda'_{c_i(a),\overline{W}}(\tilde{\psi_i}(n))\prod_{j=t+1}^{t+s}r'_{f_j,c_j(a)}(\tilde{\psi_j}(n)).
$$

Furthermore, we use the identity
$$
F'_a(n)=\prod_{i=1}^{t}\Lambda'_{c_i(a),\overline{W}}(\tilde{\psi}_i(n))+\sum_{j=1+t}^{t+s}(r'_{f_j,c_j(a)}(\tilde{\psi}_j(n))-1)F'_{a,j}(n),
$$
where $$F'_{a,j}(n)=
\prod_{k<j}
r'_{f_k,c_k(a)}(\tilde{\psi}_k(n))\prod_{i=1}^{t}\Lambda'_{c_i(a),\overline{W}}(\tilde{\psi}_i(n)).$$
%For any $b\in[W]$, we define $R_{f_j,b,W}$ on $\N$ by
%$R_{f_j,b_j,W}(n)=R_{f_j}(Wn+b_j)$.
Thus, equation \eqref{sumovercongclasses} yields
\begin{equation}
\label{sumovercongclassesII}
\left(\prod_{j=t+1}^{t+s}\frac{2\pi}{\sqrt{-D_j}}\right)^{-1}\sum_{n\in \Z^d\cap K}F(n)=T_1+T_2,\end{equation}
where
\begin{equation}
\label{T1}
T_1=\sum_{a\in [\overline{W}]^d}Q(a)\sum_{n\in 
\Z^d\cap K_a}\prod_{i=1}^{t}\Lambda'_{c_i(a),\overline{W}}(\tilde{\psi}_i(n))
\end{equation}
and
\begin{equation}
\label{T2}
T_2=\sum_{j=t+1}^{t+s}\sum_{a\in A}\sum_{n\in 
\Z^d\cap K_a}(r'_{f_j,c_j(a)}(\tilde{\psi}_j(n))-1)F'_{a,j}(n).
\end{equation}

%\prod_{k<i}
%\Lambda'_{c_k(a),W}(\tilde{\psi}_k(n))\prod_{j=t+1}^{t+s}R_{f_j,c_j(a),W}(\tilde{\psi}_j(n))
%\\
%+\left(\frac{W}{\phi(W)}\right)^t\sum_{a\in A}\sum_{n\in 
%\Z^d\cap K_a}\prod_{j=t+1}^{t+s}R_{f_j}(\psi_j(Wn+a))
%R_{f_j,c_j(a),W}(\tilde{\psi}_j(n))

Here, the first term is expected to be the main term, of the order of magnitude of $\Vol(K)$, while the second one
involving the difference of a $\overline{W}$-tricked
representation function to its average 1, is expected to be negligible, that is, $o(N^d)$. 

\subsection{Analysis of the main term}
To deal with the main term \eqref{T1}, ideally, we would like to claim that the inner sum satisfies
$$
\sum_{n\in 
\Z^d\cap K_a}\prod_{i=1}^{t}\Lambda'_{c_i(a),\overline{W}}(\tilde{\psi}_i(n))=\Vol(K_a)+o((N/\overline{W})^d).
$$
Unfortunately, this statement, proven by Green and Tao \cite[Theorem 5.1]{GT2} with $W$
instead of $\overline{W}$, is beyond reach at the moment,
basically because $\overline{W}$ is too large for the Siegel-Walfisz theorem to apply. If we were able to lower the
prime powers $p^{\iota(p)}\approx \log^{C_1}N$ involved in $\overline{W}$ to smaller prime powers $p^{\eta(p)}\approx \log\log N$, the resulting $\widetilde{W}$ would be small enough for Siegel-Walfisz (and more generally \cite[Theorem 5.1]{GT2}) to apply. 
Let us then define $\eta(p)$ by 
\begin{equation}
\label{def:eta}
p^{\eta(p)-1}<\log\log N\leq p^{\eta(p)}
\end{equation}
and $\widetilde{W}=\prod_{p\leq w}p^{\eta(p)}\leq\prod_{p\leq w}p\log\log N\ll \exp((\log\log\log N)^2)$.

The reader may wonder at this point why we performed the
$\overline{W}$ trick at all, if we really would like to 
deal with congruence classes modulo $\widetilde{W}$. The reason
for this is that Lemma~\ref{exceptional} would not hold if
$X_0$ contained all numbers smaller than $N$ that have a prime power factor larger than $\log\log N$: this is not a sparse enough set, given the possibly large values of $R_f$ and $\Lambda$. 
%However, in the local context (modulo $\overline{W}$), this will be a rare enough event, because the local counterparts of the von Mangoldt and the representation functions are much easier to bound. 
Thus, performing the $\widetilde{W}$-trick,
we could not force the residues to satisfy $c_j(a)\neq 0\mod p^{\eta(p)}$, whereas the $\overline{W}$-trick allowed us to force $c_j(a)\neq 0\mod p^{\iota(p)}$. 
Imposing such a nonzero congruence will prove crucial to ensure
that $r'_{f_j,c_j(a)}$ is dominated by a pseudorandom
majorant, and thus to ensure the term $T_2$ is negligible.
%and will again be a central feature of our argument in Section 3, because the divisors occurring in
%the tricked majorant of $R_f$ can then be assumed to only have
%prime factors $p>w(N)$. This will become clear in Appendix B.

To lower the prime powers, 
we shall rely on the powerful lift-invariance property
of Matthiesen \cite[Lemma 6.3]{Matt}.

\begin{lm}
\label{liftinvar}
Let $f$ be a PDBQF of discriminant $D$. Let $p_0$ be a prime
and $\alpha\geq v_{p_0}(d)$ be an integer. Suppose
$b\neq 0\mod p_0^{\alpha}$. Then for all $\beta\geq \alpha$ and
$c\equiv b\mod p_0^{\alpha}$, we have
$$
\rho_{f,b}(p_0^{\alpha})p_0^{-\alpha}=\rho_{f,c}(p_0^{\beta})p_0^{-\beta}.
$$
\end{lm}

To reduce prime powers, we decompose the residue set 
$[\overline{W}]^d$ into $X_1$ and $X_2$, where
$$X_1=\{a\in[\overline{W}]^d\mid \forall j \in \intnn{t+1}{t+s} \quad\forall p\leq w(N)\qquad\psi_j(a)\neq 0\mod p^{\eta(p)}\}$$
and $X_2$ is the complement of $X_1$ in $[\overline{W}]^d$. We also introduce
$$Y_1=\{a\in[\widetilde{W}]^d\mid \forall j \in \intnn{t+1}{t+s} \quad\forall p\leq w(N)\qquad\psi_j(a)\neq 0\mod p^{\eta(p)}\}.$$
First, for $a\in X_1$, we remark that $Q(a)$ depends only on the reduction $\tilde{a}\in Y_1$ of $a$. 
Indeed, writing 
$$
\widetilde{Q}(a)=\prod_{i=1}^t\Lambda_{W}(\psi_i(a))
\prod_{j=t+1}^{t+s}\frac{\rho_{f_j,\psi_j(a)}(\widetilde{W})}{\widetilde{W}}1_{\forall p\leq w,\psi_j(a)\neq 0\mod p^{\eta(p)}},
$$
we have $\widetilde{Q}(\tilde{a})=Q(a)$.
This shows that
\begin{equation}
\label{sumX1}
\begin{split}
\sum_{a\in X_1}Q(a)\sum_{n\in 
\Z^d\cap K_a}\prod_{i=1}^{t}\Lambda'_{c_i(a),\overline{W}}(\tilde{\psi}_i(n)) &=\sum_{a\in Y_1}Q(a)\sum_{\substack{b\in [\overline{W}]^d\\b\equiv a\mod \widetilde{W}}}\sum_{n\in 
\Z^d\cap K_b}\prod_{i=1}^{t}\Lambda'_{c_i(b),\overline{W}}(\tilde{\psi}_i(n))\\
&=
\sum_{a\in Y_1}Q(a)\sum_{n\in 
\Z^d\cap K_a}\prod_{i=1}^{t}\Lambda'_{c_i(a),\widetilde{W}}(\tilde{\psi}_i(n))
.
\end{split}
\end{equation}
We admit a slight abuse of notation: in the last term, $\tilde{\psi}_i$ may be different from the other occurrences of $\tilde{\psi}_i$ (differing at most in the constant term) and $c_i(a)\equiv \psi_i(a)\mod \widetilde{W}$ lies in $[\widetilde{W}]$.
They satisfy $\psi_i(\widetilde{W}n+a)=\widetilde{W}\tilde{\psi}_i(n)+c_i(a)$.
Now we claim that an asymptotic for the inner sum follows from the work of Green and Tao \cite{GT2}. To check this, notice that the properties
of $W=W(N)$ that are used there to prove Theorem 5.1 are the following.
\begin{itemize}
\item There is a function $w(N)$ tending to infinity such that every prime $p\leq w(N)$ divides $W$. This is still the case
for $\widetilde{W}$, as one can easily check that $\eta(p)\geq 1$ for all $p\leq \log\log\log N$.
\item The exceptional primes for the system $\widetilde{\Psi}$,
that is primes $p$ modulo which two forms of the systems are
affinely dependent, are $O(w)=O(\log\log N)$; this is still true in our setting. This bound was important in the application of Theorem D.3 to prove Proposition 6.4 in \cite{GT2}.
\item The size of $W$ is reasonable, namely $W=O(\log N)$; this is crucial to derive equation (12.9) from equation (12.10) in \cite{GT2}.
We also have this bound for $\widetilde{W}$.
\end{itemize}
Thus for any $a\in Y_1$ that has a nonzero contribution, in
particular, satisfying $(c_i(a),W)=1$ for all $i\in[t]$, we get
$$
\sum_{n\in 
\Z^d\cap K_a}\prod_{i=1}^{t}\Lambda'_{c_i(a),\widetilde{W}}(\tilde{\psi}_i(n))=\Vol(K_a)+o((N/\widetilde{W})^d).
$$
Inserting this equality in \eqref{sumX1} yields
$$
\sum_{a\in X_1}Q(a)\sum_{n\in 
\Z^d\cap K_a}\prod_{i=1}^{t}\Lambda'_{c_i(a),\overline{W}}(\tilde{\psi}_i(n))=(\Vol(K)+o(N^d))\E_{a\in[\widetilde{W}]^d}
Q(a)1_{a\in Y_1}.
$$
We exploit multiplicativity to write
$$
\E_{a\in[\widetilde{W}]^d}
Q(a)1_{a\in Y_1}=\prod_{p\leq w}\E_{a\in \Zp{\eta(p)}{d}}
\prod_{i=1}^t\Lambda_p(\psi_i(a))\prod_{j=t+1}^{t+s}\frac{\rho_{f_j,\psi_j(a)}(p^{\eta(p)})}{p^{\eta(p)}}1_{\psi_j(a)\neq 0\mod p^{\eta(p)}}.
$$
Now we invoke results from the Appendix A to conclude. Indeed, replacing $\iota(p)$ by $\eta(p)$ in Lemma A.2, we find that
$$
\E_{a\in \Zp{\eta(p)}{d}}
\prod_{i=1}^t\Lambda_p(\psi_i(a))\prod_{j=t+1}^{t+s}\frac{\rho_{f_j,\psi_j(a)}(p^{\eta(p)})}{p^{\eta(p)}}1_{\psi_j(a)\neq 0\mod p^{\eta(p)}}=\beta_p+O((\log\log N)^{-1/3}).
$$
Using Lemma A.3, we conclude that
$$
\E_{a\in[\widetilde{W}]^d}
Q(a)1_{a\in Y_1}
%=(1+O(\log\log\log N\cdot (\log\log)^{-1/5}N))\prod_p\beta_p=
=\prod_p\beta_p+o(1)
$$
 and finally we can write
$$\sum_{a\in X_1}Q(a)\sum_{n\in 
\Z^d\cap K_a}\prod_{i=1}^{t}\Lambda'_{c_i(a),\overline{W}}(\tilde{\psi}_i(n))=\beta_{\infty}\prod_p\beta_p+o(N^d).$$
% Now sum over X_2
\subsection{The sum over $X_2$}
We now turn to the sum over $X_2$,  which we would like to show  is $o(N^d)$.
Lacking an asymptotic for the inner sum, we shall be content
with an upper bound. Luckily, such a bound is available, thanks to a majorant of the von Mangoldt function devised by Goldston and Y\i{}ld\i{}r\i{}m; see \cite[Appendix D]{GT2} and the references therein.
Let us introduce
\begin{equation}
\label{lambdachir}
\Lambda_{\chi,R}(n)=\log R\left(\sum_{\ell\mid n}\mu (\ell)\chi\left(\frac{\log \ell}{\log R}\right)\right)^2,
\end{equation}
where $R=N^{\gamma}$ and $\gamma$ is to be chosen later, and $\chi$ is a smooth even function $\R \rightarrow [0,1]$ supported on $[-1,1]$ satisfying
$\chi(0)=1$ and $\int_0^1\chi'(x)^2dx=1$.
Finally, for any $b\in [\overline{W}]$
coprime to $\overline{W}$,
we define the Green-Tao majorant
$$\nu_{\mathrm{GT},b}:n\mapsto \frac{\phi(W)}{W}\Lambda_{\chi,R}(\overline{W}n+b).$$ The
following lemma shows that this function majorises the $\overline{W}$-tricked
von Mangoldt function.
 
\begin{lm}\label{nugt}
%\begin{enumerate}
%\item 
For any $b\in [\overline{W}]$ coprime to $W$, 
%provided that
%$\gamma <3/5$
we have
$$
\Lambda'_{b,\overline{W}}(n)\ll \nu_{\mathrm{GT},b}(n)
$$for $n\in [R,N']$, where the implied constant depends
only on $\gamma$. 
\end{lm}
\begin{proof}
%\begin{enumerate}
%\item 
To prove the first part of the lemma,
we have to take care only of the integers $n\in [R,N']$
such that $\overline{W}n+b$ is prime.
In this case, the left-hand side is bounded above by a
constant multiple of
$\frac{\phi(W)}{W}\log N$
while the right-hand side is $\frac{\phi(W)}{W}\log R$,
but $\log R=\gamma \log N\gg\log N$.
%\end{enumerate}
\end{proof}
Notice that the bound is in fact valid on $[N']$, because
if $n\leq R=N^{\gamma}$, we obviously have $\overline{W}n+b\leq N^{\gamma}(\overline{W}+1)<N^{2\gamma}$ because $N$ is large enough, and $\Lambda'(\overline{W}n+b)=0$
by the definition of $\Lambda'$ (see Subsection 2.1).

Now if $\gamma$ is small enough, $\nu_{\mathrm{GT},b}$ is known to satisfy the linear forms condition. This was shown by Green and Tao
(see \cite[Appendix D]{GT2} and the Appendix C of this paper)
with $W$ instead of $\overline{W}$, but the reader may check that in this portion of their article, the bound $W=O(\log N)$
on the size of $W$ plays no role, so that the argument works just as well with $\overline{W}$.
In particular,
$$
\sum_{n\in 
\Z^d\cap K_a}\prod_{i=1}^{t}\Lambda'_{c_i(a),\overline{W}}(\tilde{\psi}_i(n))\ll \sum_{n\in 
\Z^d\cap K_a}\prod_{i=1}^{t}\nu_{\mathrm{GT},c_i(a)}(\tilde{\psi}_i(n))=\Vol(K_a)+o((N/\overline{W})^d).
$$
Here we used crucially the fact that no two of the forms $\tilde{\psi}_i$ are rational multiple of one another; this follows from the finite complexity assumption on the original system $\Psi$. From this, we infer that
$$
\sum_{a\in X_2}Q(a)\sum_{n\in 
\Z^d\cap K_a}\prod_{i=1}^{t}\Lambda'_{c_i(a),\overline{W}}(\tilde{\psi}_i(n))\ll N^d\E_{a\in[\overline{W}]^d}Q(a)1_{a\in X_2}.
$$
We use the triangle inequality to bound the inner expectation by
$$\sum_{\substack{p\leq w\\j\in \intnn{t+1}{t+s}}}\E_{a\in [\overline{W}]^d}1_{p^{\eta(p)}\mid\psi_j(a)}Q(a),
$$
which, by multiplicativity, can be rewritten as
$$\sum_{\substack{q\leq w\\j\in \intnn{t+1}{t+s}}}\E_{a\in (\Z/q^{\iota(q)}\Z)^d}1_{q^{\eta(q)}\mid\psi_j(a)}Q_q(a)\prod_{p\leq w,q\neq p}\E_{a\in\Zp{\iota(p)}{d}}Q_p(a).$$
Here, as the reader can guess, we introduced
$$
Q_p(a)=\prod_{i=1}^t\Lambda_{p}(\psi_i(a))
\prod_{j=t+1}^{t+s}\frac{\rho_{f_j,\psi_j(a_j)}(p^{\iota(p)})}{p^{\iota(p)}}1_{\psi_j(a)\neq 0\mod p^{\iota(p)}},
$$
for any prime $p$, so that $Q(a)=\prod_{p\leq w(N)}Q_p(a)$. Again, we invoke the results of Appendix A.
Lemmas A.2 and A.3 imply that
$$
\prod_{p\leq w,q\neq p}\E_{a\in\Zp{\iota(p)}{d}}Q_p(a)=O(1)
$$
while the proof of Lemma A.1 shows that
$$
\E_{a\in (\Z/q^{\iota(q)}\Z)^d}1_{q^{\eta(q)}\mid\psi_j(a)}Q_q(a)=O((\log\log N)^{-1/3}).
$$
Because $w(N)=\log\log\log N$ is so small, we get as desired
$$\sum_{a\in X_2}Q(a)\sum_{n\in 
\Z^d\cap K_a}\prod_{i=1}^{t}\Lambda'_{c_i(a),\overline{W}}(\tilde{\psi}_i(n))=o(N^d).$$
\subsection{Reduction of the main theorem}
Given the above discussion, the main theorem (Theorem \ref{mytrm}) boils down to proving that the term $T_2$ defined
in equation \eqref{T2} is $o(N^d)$. This is a consequence of the next proposition. 
\begin{trm}
Let $d,t$ and $s$ be nonnegative integers,
and let $f_0,f_{t+1},\ldots,f_{t+s}$ be PDBQF. Let $N'= N/\overline{W}$, 
and $\Phi=(\phi_0,\ldots,\phi_{1+t+s})$ be a system of
affine-linear forms $\Z^d\rightarrow \Z^{t+s+1}$ of finite complexity whose linear coefficients are bounded
by a constant. Let $L\subset [0,N']^d$ be a convex set such that
$\Phi(L)\subset [1,N']^{t+s+1}$. Then for any $b\in B_{t,s+1}$,
we have
$$
\sum_{n\in\Z^d\cap L}(r'_{f_0,b_0}(\phi_0(n))-1)\prod_{i\in [t]}\Lambda'_{b_i,\overline{W}}(\phi_i(n))\prod_{j=t+1}^{t+s}r'_{f_j,b_j}(\phi_j(n))=o(N'^d).
$$
\label{reduction}
\end{trm}
The set $B=B_{t,s+1}$ was introduced in Definition~\ref{bts}.
Notice the slight twist of notation with respect to the original definition, due to the fact that our quadratic forms
are labelled $f_0,\ldots,f_{t+s+1}$.

We prove this theorem in the next section.

\section{Proof of Theorem \ref{reduction}}
\subsection{Generalised von Neumann theorem and uniformity}
Here and in the rest of the paper, $N'=N/\overline{W}$. To prove Theorem \ref{reduction}, we have to show that the average
along a linear system
of a product is $o(N'^d)$, knowing that one of the factor has average
$o(N'^d)$.
To do so, we reduce to a family of standard linear systems,
the ones which underlie the definition of Gowers norms
which we now introduce.
\begin{dfn}
Let $g:\Z\rightarrow\R$ be a function and $k\geq 1$ an integer.
The \emph{Gowers norm} or $U^k$ norm of $g$ on $[N]$ is the expression
$$\nor{g}_{U^k[N]}=\left(\E_{x\in[N]}\E_{h\in[N]^k}\prod_{\omega
\in\{0,1\}^k}g(x+\omega\cdot h)\right)^{2^{-k}}.$$
\end{dfn}
%In order to prove that the multilinear average
%of Theorem \ref{reduction} is small, we shall rely any integer $k$,
%that both $\nor{\Lambda'_{b,\overline{W}}-1}_{U^k[N]}=o(1)$ (for
%any $b$ coprime to $W$)
%and $\nor{r'_{f,b}-1}_{U^k[N]}=o(1)$ for any PDBQF $f$ and $b$ 
%representable by $f$ modulo $\overline{W}$ and not congruent
%to 0 modulo $p^{\iota(p)}$ for any $p\leq w(N)$ (these are the
%results of \cite{GT2} and of \cite{Matt}, respectively).
%Then we need to ensure that these $W$-tricked functions are also
%bounded by a pseudorandom measure; this was again done in the same references.
%Finally we invoke the generalized von Neumann theorem which shows that the conditions on the Gowers norms and the domination
%by a pseudorandom measure guarantee that any finite complexity
%multilinear average is arbitrarily small.
We need one more definition.
\begin{dfn}Let $D\geq 1$.
A $D$-\textit{pseudorandom measure} is a sequence of functions $\nu=\nu_M: \Z/M\Z\rightarrow \R_+$, satisfying\footnote{In earlier works such as \cite{GT2} or \cite{Matt},
there was a \emph{correlation condition}, but it is no longer necessary due to the work of Fox, Conlon and Zhao 
\cite{FCZ}, and its
integration by Tao and Ziegler \cite{TaoZieg2}.} 
\begin{enumerate}
\item $\E_{n\leq M}\nu(n)=1+o(1)$.
\item ($D$-linear forms conditions) Let $1\leq d,t\leq D$. For every finite-complexity system of affine-linear forms $\Psi : \Z^d\rightarrow \Z^t$ with coefficients bounded by $D$ and any convex set $K\subset [-M,M]^d$ such that $\Psi(K)\subset [M]^t$, the following estimate holds
\begin{equation}
\label{defeqlinforms}
\E_{n\in\Z^d\cap K}\prod \nu(\psi_i(n))=1+o(1).\end{equation}
\end{enumerate}
\end{dfn}
%\begin{rmk}
Pseudorandom measures are defined on cyclic groups
rather than intervals of integers, so the values of
the linear forms $\psi_i(n)$ are understood modulo $M$. 
Similarly, some authors prefer to define the Gowers norms on 
cyclic groups \cite[Appendix B]{GT2}, and then in intervals of integers by embedding
them in cyclic groups. However, the uniformity conditions and the von Neumann theorem
will also work well with the definition above.

On the other hand, the functions we want to majorise, of the form
$\Lambda'_{b,\overline{W}}-1$ and $r'_{f,b}-1$,
are naturally defined on an interval $[N']$.
When $N'\leq M$ are integers and $f : [N']\rightarrow \R$ is a function, we embed $[N']$ in $\Z/M\Z$ canonically
and extend the definition of $f$ to $\Z/M\Z$ by setting $f=0$ outside $[N']$.
Passing from intervals to cyclic groups involves some technicalities.
To avoid torsion
and wrap-around issues, we try to embed an interval into a sufficiently large
cyclic group of prime order
%\end{rmk}
%\begin{rmk}

%It used to crop up in the derivation of the generalised 
%Gowers inverse theorem (i.e. applied to unbounded functions)
%from the standard Gowers inverse theorem.
%\end{rmk}

Following \cite[Proposition 7.1]{GT2} ,
we can now state the generalised von Neumann theorem.
\begin{trm}
\label{vonNeumann}
Let $t,d,L$ be positive integer parameters. Then there are positive constants $1\leq \Gamma$ and $D$, depending on $t,d$ and $L$ such that
the following holds. Let $C$ be a constant satisfying $\Gamma\leq C\leq O_{t,d,L}(1)$ be arbitrary and suppose that $M\in [CN',2CN']$ is a prime. Let $\nu : \Z/M\Z
\rightarrow \R^+$ be a $D$-pseudorandom measure, and suppose
that $f_1,\ldots,f_t : [N']\rightarrow\R$
are functions with $\abs{f_i(x)}\leq \nu (x)$
for all $i\in[t]$ and $x\in [N']$.
Suppose that $\Psi=(\psi_1,\ldots,\psi_t)$ is system of affine-linear forms of finite complexity whose linear coefficients are bounded by $L$. 
Let $K\subset [-M,M]^d$ be a convex set such that $\Psi(K)\subset [M]^t$.
Finally, suppose that 
\begin{equation}
\min_{1\leq j\leq t}\nor{f_j}_{U^{t-1}[N]}=o(1).
\label{uniformity}
\end{equation}
Then we have
$$
\E_{n\in K\cap\Z^d}\prod_{i\in[t]}f_i(\psi_i(n))=o(1).
$$
\end{trm}
%The constant $C_1$ is not to be mistaken with the one
%in Definition~\ref{X0}! The quite distinct contexts should
%discard any ambiguity.

We highlight that this theorem actually replaces
a linear system $\Psi$ with another one, the system
$(x+\omega\cdot h)_{\omega\in\{0,1\}^{t-1}}$, so that
it is not immediately obvious that we have reduced the difficulty. However, it happens that uniformity with
respect to this system can be characterised in another way:
this is the inverse theorem for the Gowers norms \cite{GI}.
%\begin{rmk}
%%\begin{enumerate}
%%\item 
%%\item 
%Here we demand for at least one of the functions a very high uniformity, indeed a uniformity of order $t-1$ where $t$ is the number of linear forms. We sometimes do not need this much uniformity; the notion of complexity of a system aims at
%providing the degree of uniformity actually needed.
%In this paper, we do not go into these subtleties.
%%\item 
%%\end{enumerate}
%\end{rmk}
The following proposition provides the uniformity condition \eqref{uniformity} for our functions.
\begin{prop}\label{prop:uniform}
%\begin{enumerate}
%\item 
%\item 
Let $f$ be a PDBQF, and $b\in [\overline{W}]$ be representable by $f$ modulo $\overline{W}$
and not divisible by any $p^{\iota(p)}$ for $p\leq w(N)$.
Then the tricked representation function of $f$ defined by
\eqref{Wtrickedrep}
satisfies
$$
\forall k\in\N\qquad \nor{r'_{f,b}-1}_{U^k[N']}=o(1).
$$
%\end{enumerate}
\end{prop}
%\begin{proof}
The proof of this proposition \cite[Sections 14-18]{Matt}
consists in evaluating the correlation of $r'_{f,b}-1$ with nilsequences.

%Heuristically, it is clear that $\overline{W}$ and $W$ are just as good: they remove the irregularity of the distribution of primes in APs of $w(N)$-smooth modulus (modulus $q_1$ having only
%prime factors smaller than $w(N)$). Thus $\Lambda'_{b,\overline{W}}$ as well
%as $\Lambda'_{b,W}$ have average 1 in such APs. This appeared
%earlier as a consequence of $\frac{\phi(\overline{W}q_1)}{\overline{W}q_1}=\frac{\phi(W)}{W}$.

\subsection{Construction of a pseudorandom majorant}
We now construct a pseudorandom measure which dominates both
the $\overline{W}$-tricked von Mangoldt and representation functions.
%We recall for each an already well-studied pseudorandom majorant, and later try to combine them.
\subsubsection{The pseudorandom majorant of the von Mangoldt function.}
\label{majorantgt}
We first recall the pseudorandom majorant $\nu_{\mathrm{GT},b}(n)$ from the Green-Tao machinery, first used by Goldston and Y\i ld\i r\i m.
We already defined it in equation \eqref{lambdachir}. 
Green and Tao \cite[Lemma 9.7]{GT1} proved that it has average $1+o(1)$.
\begin{lm}
\label{nugtave}
If $\gamma$ is small enough,
for any $b\in[\overline{W}]$ coprime to $\overline{W}$, we have
\begin{align*}
\E_{n\in[N']}\nu_{\mathrm{GT},b}(n)
%&=\log R\frac{\phi(W)}{W}\sum_{l_i,l_i'}\mu(l_i)\mu(l_i')
%\chi\left(\frac{\log l_i}{\log R}\right)
%\chi\left(\frac{\log l_i'}{\log R}\right)
%\E_{n\in[N']}1_{l_i\mid \overline{W}n+b}
%1_{l_i'\mid\overline{W}n+b}\\
%&
=1+o(1).
\end{align*}
%\end{enumerate}
\end{lm}

\subsubsection{The pseudorandom majorant of the representation function}
\label{majorantmatt}
We also use a pseudorandom majorant from Matthiesen's work. 
For this we need to recall some notation and facts from \cite{Matt}.
For a set $\A$ of primes, $\langle\A\rangle$ stands for the set
of integers whose prime factors are all in $\A$,
and
$\tau_D(n)=\sum_{d\in \langle\P_D\rangle}1_{d\mid n}$.
\begin{prop}
For any integer $D\equiv 0,1\mod 4$, 
there exists a
set of primes $\P_D$ of density 1/2, which is a union of
congruence classes modulo $D$, such that putting
$\P_D^*=\P_D\cup\{p\in \P\colon p\mid D\}$ and
$\Q_D=\P\setminus\P_D^*$,
we have, for any PDBQF $f$ of discriminant $D$, the bound
$$
R_f(n)\ll_D \tau_D(n)\sum_{\substack{m\in \langle\Q_D\rangle\\
m^2\mid n}}1_{\langle\P_D^*\rangle }(n/m^2).
$$
\label{ReprMaj}
\end{prop}
To understand this result heuristically, which is
the starting point of the construction of the pseudorandom
majorant in \cite{Matt}, we recall that the number of representations
of any odd number $n$ as a sum of two squares
is $4\sum_{d\mid n}\chi(d)$ where $\chi$ is the only nontrivial character modulo 4. By multiplicativity, this is easily seen
to become $\tau_4(n)\prod_{p\equiv 3\mod 4}1_{v_p(n)\equiv 0\mod 2}$, with $\P_4$ being the set of primes congruent to 1 modulo 4, from which we derive a majorant of the desired form. This works similarly for other quadratic forms.

Thus, to majorise the function $R_f$ it will be enough to
majorise the functions $\tau_D$ and
$1_{\langle\P_D^*\rangle}$. 
%let us first deal with $\tau_D$. 
The heuristic to bound $\tau_D$ (or rather $\tau_D/\sqrt{\log N}$) is as follows (see \cite[Lemma 4.1]{Matt2}).
We would like to truncate the divisor sum defining it
at $N^{\gamma}$ (possibly with a smooth cut-off), just as was done
earlier for the von Mangoldt function. The function defined by this truncated divisor sum is called $\tau_{\gamma}$. Unfortunately, it turns out
that the inequality $\tau\leq C\tau_{\gamma}$ is not entirely true, at least not true with the same constant $C$
throughout the first $N$ integers. Nevertheless, a heuristic of Erd\H{o}s \cite{Erdos} says that an integer is either excessively rough
or excessively smooth or has a cluster of many prime factors close together. Moreover we have excluded the 
first two possibilities when we took out the set $X_0$, so it remains to majorise $\tau(n)$ in the third case. Then the
bound depends on the position of this cluster of primes and on its density. For more detail on the majorant of the
divisor function, see \cite{Matt2}.

To bound $1_{\langle\P_{D}^*\rangle}$ (or rather $1_{\langle\P_{D}^*\rangle}\sqrt{\log N}$), that is, the indicator
function of the integers without any prime factor belonging to  $\Q_D$,
we use a sieving-type majorant, that is, a majorant similar to the one introduced above for the von Mangoldt function.
Indeed, integers  without any prime factor in $\Q_D$ are similar to prime numbers (integers without any 
non-trivial prime factor at all).

To formalise this heuristic, let us introduce the
following definition.
\begin{dfn}
\label{setsU}
Let $\xi=2^{-m}$ for some $m\in\N$, put $\gamma=2\xi$.
We define sets $U(i,s)$ for integers $i,s$ as follows.
For $i=\log_2(2/\xi)-2=m-1$, we let $U(i,2/\xi)$ be $\{1\}$
and otherwise $U(i,2/\xi)=\emptyset$.
If $s>2/\xi$, write $U(i,s)$ for the set of all products of
$m_0(i,s)=\lceil\xi s(i+3-\log_2s)/100\rceil$ distinct primes
from the interval $[N^{2^{-i-1}},N^{2^{-i}}]$.
\end{dfn}
%Then we use the following lemma.
%\begin{lm}
%Given the above definition of the set $X_0$ and any integral affine-linear form
%$\psi : \Z^d\rightarrow \Z$ and a convex set $K\subset [-N,N]^d$
%with $\psi(K)\subset [N]$ we have
%$$
%\E_{n\in\Z^d\cap K}1_{\psi(n)\in X_0}\ll \log^{-C_1/2}
%$$
%\end{lm}
Let us fix an integer $D\equiv 0,1\mod 4$.
We now propose a majorant for the W-tricked representation function of a PDBQF of discriminant $D$, which was designed by Matthiesen \cite{Matt}. We again need the smooth function $\chi$ (this should not be mistaken with a character, as there are no more characters in the sequel) introduced
for the majorant of the von Mangoldt function.
We use the function
\begin{equation}
\label{rdgamma}
r_{D,\gamma}(n)=\frac{\beta'_{D,\gamma}(n)\nu'_{D,\gamma}(n)}{C_{D,\gamma}},
\end{equation}
where $$\nu'_{D,\gamma}=\sum_{s=2/\xi}^{(\log\log N)^3}\sum_{i=\log_2s-2}^{6\log\log\log N}\sum_{u\in U(i,s)}2^s1_{u\mid n}\tau'_{D,\gamma}(n),$$

\noindent with

$$\tau'_{D,\gamma}(n)=\sum_{\substack{d\in  \left\langle\P_D\right\rangle\\p\mid d\Rightarrow p>w(N)}}1_{d\mid n}\chi\left(\frac{\log d}{\log N^{\gamma}}\right),$$
and 
$$
\beta'_{\gamma}(n)=\sum_{\substack{m\in \left\langle\Q_D\right\rangle\\p\mid m\Rightarrow p>w(N)\\ m< N^{\gamma}}}\left(\sum_{\substack{e\in \left\langle\Q_D\right\rangle\\
p\mid e\Rightarrow p>w(N)}}1_{m^2e\mid n}\mu (e)\chi\left(\frac{\log e}{\log N^{\gamma}}\right)\right)^2.
$$
The constant $C_{D,\gamma}$ is the
one which ensures that the function $r_{D,\gamma}$ has average 1; the next lemma asserts the existence of such a constant.

We now define for any $b\in[\overline{W}]$ the function $\nu_{\mathrm{Matt},b,D}:[N']\rightarrow\R$ by
\begin{equation}
\nu_{\mathrm{Matt},b,D}(n)= r_{D,\gamma}(\overline{W}n+b).
\label{eq:defnumatt}
\end{equation}
The next lemma \cite[Lemma 7.5]{Matt} also asserts that
this function is a pseudorandom majorant for the
representation function of any PDBQF of discriminant $D$.
Recall Definition \ref{bts}.
%Symmetrically to the previous subsection, we prove
%that this construction provides a majorant of average 1 to the tricked representation function.
\begin{lm}\label{numatt}
%\begin{enumerate}
%\item
For any PDBQF $f$ of discriminant $D$ and $b\in[\overline{W}]$
satisfying $b\neq 0\mod p^{\iota(p)}$ for any
$p\leq w(N)$ and  $\rho_{f,b}(\overline{W})>0$,
the following bound holds
$$
r'_{f,b}(n)\ll \nu_{\mathrm{Matt},b,D}(n).
$$
%\item 
Furthermore, for some $C_{D,\gamma}=O(1)$, we have
$\E_{n\in[N']}\nu_{\mathrm{Matt},b,D}(n)=1+o(1)$.
%\end{enumerate}
\end{lm}
%\begin{proof}
%This is simply Lemma 7.5 from \cite{Matt}.
%\end{proof}

%The majorant of the divisor function we have just introduced
%looks extremely complicated, and much more recently, new
%majorants have appeared, also of the form of truncated divisor sums. Let us mention the work of Munshi \cite{Munshi}, where
%inequalities of the form
%$$
%\tau(n)\ll\sum_{d\leq n^{\delta}}\tau(d)^{\beta} 
%$$
%for any $\delta$ with $\beta=\beta(\delta)$ are discussed.
%Unfortunately, the right-hand side has an average of
%size $\log ^{C(\delta)}N$ up to $N$, which is much larger than the average of the left-hand side, which is asymptotic to $\log N$. Thus, such majorants cannot be used as pseudorandom measures and we are left with the sixty years old idea Erd\H{o}s.

The crucial property of $\nu_{\mathrm{Matt}}$ is that
it is a truncated divisor sum, like $\nu_{\mathrm{GT}}$. Indeed, all divisors appearing
are constrained to be less than $R=N^{\gamma}$. It is obvious by definition of $\chi$
for the divisors called $d,m,e$ and less obvious, but proven by
Matthiesen, for $u$ (see Remark 3 following Proposition 4.2 in \cite{Matt}). Moreover, the divisors $d,m,e$ only have prime factors larger than $w(N)$ (this feature is also present in $\Lambda_{\chi,R}$), while $u$ has only prime factors larger than $N^{(\log\log N)^{-3}}$.
\subsubsection{Combination of both majorants}
\label{majorant}
To be able to use the von Neumann theorem (Theorem
\ref{vonNeumann}), and thus establish Theorem \ref{reduction}, we need to bound
all $t+s+1$ functions by the same majorant. Now each of them
is bounded individually by some pseudorandom majorant defined above,
so we define our common majorant by averaging all these majorants.
Recall that $N'=N/\overline{W}$ ; we take
$M$ to be a prime satisfying $N'<M\leq O(N')$.
Given a family $f_0,f_{t+1},\ldots, f_{t+s} $ of PDBQF of
discriminants $D_0,D_{t+1},\ldots,D_{t+s}$ and a family $(b_0,\ldots,b_{t+s})\in B$,
we define a function $\nu^*$ on $[N']\subset\Z/M\Z$ by
\begin{equation}
\label{eqdef:nustar}
\nu^*(n)=\frac{1}{t+s+2}(1+\sum_{i=1}^t\nu_{\mathrm{GT},b_i}(n)+\sum_{j=t+1}^{t+s}
     \nu_{\mathrm{Matt},b_j,D_j}(n)+\nu_{\mathrm{Matt},b_0,D_0}(n)).
\end{equation}
We extend it to $\Z/M\Z$ by setting $\nu^*(n)=1$ outside $[N']$.
% There are four functions,
%of the type $\Lambda'_{b_i,\overline{W}}$ (with a +1 for one of them) and $r'_{b_4,\overline{W}}$
%bounded by four pseudo-random measures $\phi(W)/W\Lambda_{\chi,R}(\overline{W}n+b_i)$ and
%$r_{\chi,\gamma}(\overline{W}n+b_4)$. We would like to show that the pointwise average of these four functions is also a pseudo-random measure, which would then be a pseudo-random majorant to all of them.
%
%In fact, rigorously we have to take some $M$ satisfying $C_2N'\leq
%M\leq C_3N'$, and then define 
%$\nu^*$ on $\Z/M\Z$ as follows
%\begin{equation}
%\label{eqdef:nustar}
%\nu^*(n)=
%\left \{
%\begin{array}{c @{\text{ if }} c}
%     \frac{1}{t+s+1}(1+\frac{\phi(W)}{W}\sum_{i=1}^t\nu_{\mathrm{GT},b_i}(n)+\sum_{j=t+1}^{t+s}
%     %r_{D_j,\gamma}(\overline{W}n+b_j) 
%     \nu_{\mathrm{Matt},b_j,D_j}(n))\quad
%     & n\leq N' \\
%    1 & N'<n\leq M \\
%\end{array}
%\right.
%\end{equation}
%\begin{rmk}
Our strategy of forming a common majorant for a family of functions by averaging a family
of majorants is not 
really unheard of. In fact, Green and Tao \cite{GT2} 
had to combine
the majorants $n\mapsto \Lambda_{\chi,R}(\overline{W}n+b_j)$ for various $b_j$ and so did Matthiesen \cite{Matt}.
Notice also that L{\^e} and Wolf \cite{Wolf} devised
a certain condition of compatibility for
two pseudorandom majorants.
However, in our case the majorants have rather different origins. But they have a similar structure,
the structure of a truncated divisor sum, so that
the proof of the linear forms condition will
not be much harder than the ones in \cite{GT2} or
\cite{Matt}.
%\end{enumerate}
%\end{rmk}

We observe that $\nu^*$ satisfies
$$
1+\sum_{i=1}^t\Lambda'_{b_i,\overline{W}}
+\sum_{j=t+1}^{t+s}r'_{f_j,b_j}+r'_{f_0,b_0}\ll\nu^*
$$
and has average $1+o(1)$
by Lemmas \ref{nugtave} and \ref{numatt}. So to ensure that $\nu^*$ is a pseudorandom measure, it remains only
to prove the linear forms condition \eqref{defeqlinforms}.
%(at least on the bulk of $[N']$, so $[N'^{3/5}]$ and $X_0$ excepted) 
This is the content of the next proposition.
\begin{prop}\label{nustarlinforms}
Fix a constant $D>0$, and positive integers $t,s$. Then there exists a constant $C_0(D)$ such
that the following holds. 
For any  bounded $C\geq C_0(D)$
there exists $\gamma=\gamma(C,D)$ such that if $M\in [CN',2CN']$
is a prime, 
$b\in B_{t,s+1}$ and $f_0,f_{t+1},\ldots, f_{t+s}$
are PDBQF 
and $\nu^*$ is defined as above, then
$\nu^*$ satisfies the $D$-linear forms condition and
for any $i\in [t]$ we have 
$$ \Lambda'_{\overline{W},b_i}
\ll \nu^*.$$ 
Similarly, we have
$$\abs{r'_{f_0,b_0}-1}\ll \nu^*$$ and
for any $j\in\intnn{t+1}{t+s}$,
we have $$r'_{f_j,b_j}\ll \nu^*$$ 
where all inequalities are valid on $[N']$.
\end{prop}
%\begin{proof}
The inequalities have already been observed above. The
linear forms condition will follow from the following
proposition.
\begin{prop}
Let $1\leq d,t,s\leq D$, where $D$ is the constant appearing in
Theorem~\ref{vonNeumann}. For any finite complexity system
$\Psi : \Z^d\rightarrow\Z^{t+s}$ whose linear coefficients are bounded by $D$
and every convex $K\subset [0,N]^d$ such that $\Psi(K)\subset [1,N/\overline{W}]^t$, and any $b\in B$, the estimate
\begin{multline}
%\begin{split}
\E_{n\in\Z^d\cap K}
\prod_{j=t+1}^{t+s}r_{D_j,\gamma}(\overline{W}\psi_j(n)+b_j)\prod_{i\in[t]}
\frac{\phi(W)}{W}\Lambda_{\chi,R}(\overline{W}\psi_i(n)+b_i)
\\
=1+O_D\left(
\frac{N^{d-1+O_D(\gamma)}}{\Vol (K)}\right)+o_D(1)
%\end{split}
\label{eqlinforms}
\end{multline}
holds, provided $\gamma$ is small enough.\label{linforms}
\end{prop}
Notice that the $t$ and $s$ are not the same as in Proposition
\ref{nustarlinforms}. The proof is postponed to Appendix B.

Deriving the linear forms conditions for $\nu^*$ (Proposition
\ref{nustarlinforms}) from Proposition \ref{linforms} requires some extra work, because of
the piecewise definition of $\nu^*$.
This was done in \cite[Proposition 9.8]{GT1} for instance, but
see also \cite[Proposition 8.4]{FCZ}, where the same ``localisation argument" is employed. Matthiesen
also relies on it in \cite{Matt}. The argument does not need any modification, so we do not reproduce it
here and invite the reader to consult
one of the references.
%\end{proof}
We can now prove Theorem~\ref{reduction}.
\begin{proof}[Proof of Theorem~\ref{reduction} assuming Proposition \ref{nustarlinforms}]
Take any integers $d,t$ and $s$, and
a system $\Phi : \Z^d\rightarrow \Z^{t+s+1}$ of affine-linear
forms of finite complexity, where the coefficients
of the linear part are bounded by $L$ and take $f_0,f_{t+1},\ldots,f_{t+s}$
any PDBQF.
Take a convex set $K\subset [1,N']^d$ such that $\Phi(K)\subset [N']^{t+s+1}$. Let $b\in B$.
Then Proposition \ref{vonNeumann} and Proposition \ref{nustarlinforms} provide constants $C_0$ and $\Gamma$, of which we take the maximum $C=\max (C_0,\Gamma)$.
Now take a prime $M\in [CN',2CN']$. Such a prime exists by Bertrand's postulate. Define $\nu^*$ as above \eqref{eqdef:nustar}.
Define $F_0=r'_{f_0,b_0}-1$.
Put $F_i=\Lambda'_{b_i,\overline{W}}$ for $i\in[t]$ and
$F_j=r'_{f_j,b_j}$ for $j\in \{t+1,\ldots,t+s\}$.
Then we have that $\abs{F_j}\ll \nu^*$ for all $j\in\{0,\ldots,t+s\}$ and
$\nu^*$ is a pseudorandom measure by Proposition \ref{nustarlinforms}, so that we can invoke the von Neumann theorem (Theorem \ref{vonNeumann}). Together with the
statements of Proposition~\ref{prop:uniform} (specialised to $k=t+s$), it implies Theorem~\ref{reduction}.
\end{proof}
%\begin{rmk}
%\begin{enumerate}
%\item
We remark that although we want to prove a result
concerning quadratic and not linear patterns in the primes,
we do not need the polynomial forms condition introduced in
\cite{TaoZieg}.
This is because the polynomial character of our configurations
is encapsulated in the representation functions
of the quadratic forms.
%\end{rmk}
%\item 

We have completed the proof our main theorem, conditionally
on the following rather technical appendices.
Appendix A provides estimates concerning the local factors
that were used in Section 2.
In Appendix B, we check the linear forms
condition for the majorant introduced above, that is, we
prove Proposition \ref{nustarlinforms}.
Appendix C provides elementary justifications to some statements
made in Appendices A and B.
\appendix
\section{Analysis of the local factors $\beta_p$}
\label{sec:locfact}
First, we check that the limit defining $\beta_p$ in Theorem \ref{mytrm} exists. 
We fix integers
$d,t,s\geq 1$ and a system of linear forms $\Psi : \Z^d\rightarrow \Z^{t+s}$
of finite complexity, and we suppose its
linear coefficients are bounded by $L$.
%We let $Z=\zeta_1,\cdots,\zeta_s$ be the system consisting of the last $s$ forms of $\Psi$, i.e. $\zeta_i=\psi_{t+i}$.

We also fix PDBQFs $f_{t+1},\ldots,f_{t+s}$ of discriminants $D_{t+1},\ldots,D_{t+s}$; 
these notions and the notation $\rho_{f_j}$
were defined in the introduction.
Let $p$ be a fixed prime
 and $m\geq 1$ an integer. For $a\in\Zp{m}{d}$, let
\begin{equation}
\label{pma}
 P_m(a)=\prod_{i=1}^t\Lambda_p(\psi_i(a))\prod_{j=t+1}^{t+s}
\frac{\rho_{f_j,\psi_j(a)}(p^m)}{p^m}.
\end{equation}
Finally, let $\beta_p(m)=\E_{a\in\Zp{m}{d}}P_m(a)$.
Thus we want to prove that $\beta_p(m)$ is convergent as
$m$ tends to $\infty$. This is a consequence of the following 
proposition
\begin{prop}
\label{prop:betapm}
The sequence $(\beta_p(m))_{m\in\N}$ is a Cauchy sequence.
More precisely, there exists $M_0=M_0(D_{t+1},\ldots,D_{t+s})$
so that for all integers $m_0\geq M_0$ and $m,n\geq m_0$, we have
$$\beta_p(m)-\beta_p(n)=O(m_0^sp^{-m_0/2})$$
\end{prop}
To facilitate the proof of this proposition and the further analysis of local factors,
we ought to introduce a 
convenient notation present in both
\cite{GT2} and \cite{Matt}.
\begin{dfn}\label{def:locdens}
For a given system of affine-linear forms
$\Psi=(\psi_1,\ldots,\psi_t):\Z^d\rightarrow\Z^t$,
positive integers $d_1,\ldots,d_t$ of lcm $m$, define the local divisor density by
$$
\alpha_{\Psi}(d_1,\ldots,d_t)=\E_{n\in (\Z/m\Z)^d}\prod_{i=1}^t
1_{\psi_i(n)\equiv 0\mod d_i}.
$$
\end{dfn}
We now prove the proposition.
\begin{proof}
%\in\intnn{t+1}{t+s}
Let $M_0=\max_{j}v_p(D_j)$ and $m_0\geq M_0$.
Let $m,n\geq m_0$. 
We split $(\Z/p^m\Z)^d$
into two parts
$$
A_1=A_1(m,m_0)=\{a\in(\Z/p^m\Z)^d\mid \forall j \in \intnn{t+1}{t+s} \quad\psi_j(a)\neq 0\mod p^{m_0}\}$$
and
$$A_2=A_2(m,m_0)=\{a\in(\Z/p^m\Z)^d\mid\exists j \in \intnn{t+1}{t+s} \quad\psi_j(a)\equiv 0\mod p^{m_0}\}.
$$
Thus
\begin{equation}
\label{decompobetap}
\beta_p(m)=\E_{a\in\Zp{m}{d}}P(a)1_{A_1(m,m_0)}(a)+\E_{a\in\Zp{m}{d}}P(a)1_{A_2(m,m_0)}(a).
\end{equation}
For the first term, we use the lift-invariance property 
\cite[Corollary 6.4]{Matt} already stated in Lemma~\ref{liftinvar}.
It implies that
$$
\E_{a\in\Zp{m}{d}}P(a)1_{A_1(m,m_0)}=\E_{a\in\Zp{m_0}{d}}P(a)1_{A_1(m_0,m_0)}
$$
thus the first term on the right-hand side
of \eqref{decompobetap} does not depend on $m$.
For the second term, we invoke
the following
general 
bound from \cite{Matt} (see Lemma 6.3 and the proof of Lemma 8.2)
$$
\frac{\rho_{f_j,\psi_j(a)}(p^{m})}{p^{m}}
\ll \sum_{k=0}^m1_{\psi_j(a)\equiv 0\mod p^k}.
$$
We also use the trivial bound $\Lambda_p\leq 2$ to infer the
inequalities
\begin{align*}
P(a)1_{A_2}(a)&\ll 2^t1_{A_2}(a)\prod_{j=t+1}^{t+s}\sum_{k=0}^{m}1_{\psi_j(a)\equiv 0\mod p^k}\\
&\ll m_0^s\sum_{\substack{0\leq k_{t+1},\ldots, k_{t+s}\leq m\\\max k_i\geq m_0}}\prod_{j=t+1}^{t+s}1_{\psi_j(a)\equiv 0\mod p^{k_j}}.
\end{align*}
Here the factor $m_0^s$ appears as the number of $s$-tuples whose entries are all in $\intnn{0}{m_0-1}$; moreover, the
$2^t$ is merged with the implied constant, which crucially
remains independent of $m$ or $m_0$.
We then average over $a$ and let
\begin{equation}
\label{systemZ}
Z=(\zeta_1,\ldots,\zeta_s)=(\psi_{t+1},\ldots,\psi_{t+s})
\end{equation} be the system of the $s$ last linear forms of $\Psi$, obtaining
\begin{equation}
\label{eq:boundlocfac}
\E_{a\in\Zp{m}{d}}P(a)1_{A_2}(a)\ll m_0^s
\sum_{\substack{0\leq k_{1},\ldots, k_{s}\leq m\\M:=\max k_i\geq m_0}}\E_{a\in\Zp{M}{d}}\prod_{i=1}^s1_{p^k_i\mid\zeta_i}.
\end{equation}
We recognise the local density $\alpha_Z$ (see Definition \ref{def:locdens}) on the right hand-side, so we put
$$
\delta_p=\sum_{\substack{0\leq k_{1},\ldots, k_{s}\leq m\\M:=\max k_i\geq m_0}}\alpha_{Z}(p^{k_1},\ldots,p^{k_s}),
$$
enabling us to rewrite \eqref{eq:boundlocfac} as
$$
\E_{a\in\Zp{m}{d}}P(a)1_{A_2}(a)\ll m_0^s
\delta_p.
$$
Since the linear coefficients of $Z$ are bounded and none of its forms is the trivial form, we see that the maximal $k$ such that $\zeta_i$ is the trivial form modulo $p^k$ is bounded.
Remark~\ref{rmk:locdens} in Appendix C, where we collect a number of elementary justifications in order not to break the flow of the exposition here, implies a bound of the form
$$\alpha_{Z}(p^{k_1},\ldots,p^{k_s})\ll p^{-\max_j k_j},$$
and thus
$$
\delta_p\ll\sum_{\substack{0\leq k_{t+1},\ldots, k_{t+s}\leq m\\M:=\max k_i\geq m_0}}p^{-M}.
$$
Bounding the number of tuples
$(k_{1},\ldots,k_{s})$ satisfying $\max k_i=j$ crudely by $(j+1)^s$, we conclude that
\begin{align*}
\delta_p &\ll \sum_{j\geq m_0}p^{-j}j^s\\
&\ll\sum_{j\geq m_0}p^{-j/2}\\
&\ll_p p^{-m_0/2}.
\end{align*}
Finally, this means that for $m\geq m_0$, we have
$$\beta_p(m)=\E_{a\in\Zp{m_0}{d}}P(a)1_{A_1(m_0,m_0)}+O(m_0^sp^{-m_0/2}).$$
The same holds for $\beta_p(n)$,
hence
$$\beta_p(m)-\beta_p(n)=O(m_0^sp^{-m_0/2})$$
and the conclusion follows.
\end{proof}

\begin{lm}
\label{lm:locfac1}
Let $p$ be a prime. Then
$$\E_{a\in(\Z/p^{\iota(p)}\Z)^d}
\prod_{i=1}^t\Lambda_p(\psi_i(a))
\prod_{j=t+1}^{t+s}\frac{\rho_{f_j,\psi_j(a)}(p^{\iota(p)})}{p^{\iota(p)}}1_{\psi_j(a)\neq 0\mod
p^{\iota(p)}}=\beta_p+O(\log ^{-C_1/3}N)
$$
where $C_1$ is the constant appearing in \eqref{def:iota},  the definition of $\iota(p)$
\end{lm}
\begin{proof}
We simply apply the proof of the above proposition with $m_0=\max(\iota(p),M_0)$. We use
$m_0^sp^{-m_0}\ll p^{-m_0/3}$, where the implied constant is
independent of $m_0$ and $p_0$.
This yields the desired result.
\end{proof}
We now analyse the behaviour of $\beta_p$ as $p$ tends to infinity.
\begin{lm}\label{lm:locfac2}
For primes $p$ tending to infinity,
$$
\beta_p=1+O(p^{-2}).
$$\end{lm}
Thus the product of the $\beta_p$ is convergent 
%(though
%it could be constantly 0 from a certain rank on, because
%some $\beta_p$ may be 0 for a small $p$) 
and
$$
\prod_{p\leq w(N)}\beta_p=\left(1+O\left(\frac{1}{w(N)}\right)\right)\prod_p\beta_p.
$$

\begin{proof}
Assume $p$ is large enough so that $p$ does not divide the product $D_{t+1}\cdots D_{t+s}$
of the (negative) discriminants of our quadratic forms.

Recall the notation $P(a)=P_m(a)$ from \eqref{pma} and the sets
$A_1=A_1(m,m)$ and $A_2=A_2(m,m)$ introduced during the proof of
Proposition~\ref{prop:betapm}.
As $m$ tends to $\infty$, we have
\begin{align*}
\beta_p +o(1) &= \E_{a\in (\Z/p^m\Z)^d}
P(a)\\
&=
\E_{a\in\Zp{m}{d}}P(a)1_{A_1}(a)+\E_{a\in\Zp{m}{d}}P(a)1_{A_2}(a)\\
&=\frac{1}{p^md}\sum_{a\in A_1}P(a)+
2^t O(sm^sp^{-m}).
\end{align*}
To get this error term, we used Corollary \ref{cor:localdens} and
the triangle inequality to bound $\abs{A_2}$, and the fact 
that $\rho_{f_j,\beta}(p^m)/p^m\ll m$ \cite[Lemma 6.3(c)]{Matt} to bound $P(a)$.
This error term tends to 0 as $m$ tends to infinity, and thus merges with the $o(1)$ of the left-hand side. 
Let us now consider the main term. Thanks to the choice of $p$ and the fact that the forms do not vanish at $a$ mod $p^m$, we can use 
Lemma 6.3 from \cite{Matt} which states that if $f$ is a PDBQF
of discriminant $D$, and if $p$ is a prime which does not divide $D$, and if $\beta\neq 0\mod p^m$, then
$$
\frac{\rho_{f,\beta}(p^m)}{p^m}=(1-\chi_D(p)p^{-1})\sum_{k=0}^m 1_{p^k\mid m}\chi_D(p^k).
$$
Here $\chi_D$ is a real character modulo $p$, namely the Kronecker symbol \cite[Lemma 2.1]{Matt}.
Thus
$$\beta_p=\lim_{m\rightarrow\infty}\E_{a\in(\Z/p^m\Z)^d}
\prod_{i=1}^t
\Lambda_p(\psi_i(a))\prod_{j=t+1}^{t+s}\left(1-\chi_{D_j}(p)p^{-1}\right)\sum_{k=0}^m1_{p^k\mid\psi_j(a)}\chi_{D_j}(p^k)$$
where we have obviously reintegrated the once excluded $a\in A_2$, because their sparsity ensures that they do not affect the limit.
For $a\in(\Z/p^m\Z)^d$, we then write $a=a'+pb$ with 
$b\in(\Z/p^{m-1}\Z)^d$ and $a'\in[p]^d$. Thus the average
$\E_a$ becomes 
\begin{equation}
\prod_{j=t+1}^{t+s}\left(1-\chi_{D_j}(p)p^{-1}\right)
\E_{a\in[p]^d}\prod_{i=1}^t
\Lambda_p(\psi_i(a))\E_{b\in(\Z/p^{m-1}\Z)^d}\prod_{j=t+1}^{t+s}
\sum_{k=0}^m 1_{p^k\mid\psi_j(a+pb)}\chi_{D_j}(p^k).
\label{bigterm}
\end{equation}
We expand the product of sums as follows
\begin{equation*}
\begin{split}
&\prod_{j=t+1}^{t+s}
\sum_{k=0}^m 1_{p^k\mid\psi_j(a+pb)}\chi_{D_j}(p^k)\\
&=
1+\sum_{j}\sum_{k_j=1}^m1_{p^{k_j}\mid\psi_j(a+pb)}\chi_{D_j}(p^{k_j})
+\sum_{\substack{0\leq k_{t+1},\ldots,k_{t+s}\leq m\\\text{ at least two }k_i>0}}\prod 1_{p^{k_j}\mid\psi_j(a+pb)}\chi_{D_j}(p^{k_j})
\end{split}
\end{equation*}
according to whether we take no, one or several nonzero $k$.
The expectation over $a$ from \eqref{bigterm} then splits into three terms.
The first one is
$$
\E_{a\in(\Z/p\Z)^d}\prod_{i=1}^t
\Lambda_p(\psi_i(a)),
$$
and he second one is
\begin{equation}
\sum_{j=t+1}^{t+s}\sum_{k_j=1}^m\chi_{D_j}(p^{k_j})\E_{a\in[p]^d}\prod_{i=1}^t
\Lambda_p(\psi_i(a))
\E_{b\in(\Z/p^{m-1}\Z)^d}1_{p^{k_j}\mid\psi_j(a+pb)}.
\label{second}
\end{equation}
Now we decompose
$\psi_j(a+pb)=\psi_j(a)+p\dot{\psi_j}(b)$, where
$\dot{\psi}$ is the linear part of $\psi$.
If $p^{k_j}$ is to divide $\psi_j(a)+p\dot{\psi_j}(b)$,
we need $p\mid\psi_j(a)$. 
Thus we can write, for each such $a$ fixed, $\psi_j(a+pb)=p\tilde{\psi_j}(b)$, where $\tilde{\psi_j}$ is 
again an affine-linear form whose linear part is $\dot{\psi_j}$.
We then need $p^{k_j-1}\mid\tilde{\psi_j}(b)$.
Because of Corollary \ref{prop:locbounds},
$$\E_{b\in(\Z/p^{m-1}\Z)^d}1_{p^{k_j-1}\mid
\tilde{\psi_j}(b)}=p^{-k_j+1}$$ 
so  the expression \eqref{second} equals
$$
\sum_{j=t+1}^{t+s}\sum_{k_j=1}^m\chi_{D_j}(p^{k_j})p^{-k_j}\E_{a\in[p]^d}\prod_{i=1}^t
\Lambda_p(\psi_i(a))p1_{p\mid\psi_j(a)}
$$
To deal with the last term,
which is
\begin{equation}
\E_{a\in(\Z/p^m\Z)^d}\prod_{i=1}^t\Lambda_p(\psi_i(a))
\sum_{\substack{0\leq k_{t+1},\ldots,k_{t+s}\leq m\\\text{ at least two }k_i>0}}\prod_{j=t+1}^{t+s} 1_{p^{k_j}\mid\psi_j(a)}\chi_{D_j}(p^{k_j}),
\label{error}
\end{equation}
we crudely bound $\Lambda_p$ by $2$ and $\chi_{D_j}$ by 1.
Recall the notation $Z$ from \eqref{systemZ}. Thus as $m$ tends
to infinity, the expression \eqref{error} is bounded above by
a constant times
$$O\left(\sum_{\substack{k_{1},\ldots,k_{s}\\\text{ at least two }k_i>0}}\alpha_Z(p^{k_{1}},\ldots,p^{k_{s}})\right)
$$
%$$
%\lim_{m\rightarrow\infty}\E_{b\in(\Z/p^{m-1}\Z)^d}\prod_{j=t+1}^{t+s}
%\sum_{k=0}^m 1_{p^k\mid\psi_j(a+pb)}\chi_{D_j}(p^k)
%=\sum_{(k_1,\ldots,k_s)\in\N^s}\alpha_{\phi_{t+1},\ldots,\phi_{t+s}}(p^{k_1},\ldots,p^{k_s})
%\prod_{j=t+1}^{t+s}\chi_{D_j}(p^{k_j})
%$$
%and we shall drop the subscript on $\alpha$. Notice the sum
%over $s$-tuples $(k_i)$ can be decomposed into a sum
%over $s$-tuples
%where at most one $k_i$ is nonzero and another one where at 
%least two of them are non-zero. We also exploit the fact
%that $\alpha(p^{k_1},\ldots,p^{k_s})=p^{-k_j}$ when $k_j$
%is the sole non-zero element in the $s$-tuple. Thus we see that
%$$
%\sum_{(k_1,\ldots,k_s)\in\N^s}\alpha(p^{k_1},\ldots,p^{k_s})
%\prod_{j=t+1}^{t+s}\chi_{D_j}(p^{k_j})
%=1+\sum_{j=t+1}^{t+s}\sum_{k_j>0}\chi_{D_j}(p^{k_j})p^{-k_j}
%+O\left(\sum_{\substack{k_{t+1},\ldots,k_{t+s}\\\text{ at least two }k_i>0}}\alpha(p^{k_1},\ldots,p^{k_s})\right)
%$$
To bound this expression, we remember that $Z$ is a system of finite complexity. This implies, thanks to Proposition
\ref{prop:locbounds}, that for $p$ large enough\footnote{We need $p$ to be large because for
some small $p$, there could be two forms that, though affinely
independent, become dependent when reduced modulo $p$. Such primes are called \emph{exceptional}.
The same need for large $p$ will appear again later.
}
depending on $s,d,L$, we have
$$
\alpha_Z(p^{k_{1}},\ldots,p^{k_{s}})\leq p^{-\max_{i\neq j}(k_i+k_j)}
\leq p^{-1-\max(k_i)}
$$
whenever at least two $k_i$ are nonzero.
For any $k\geq 1$, there are at most $s(k+1)^{s-1}$ $s$-tuples
that satisfy $\max k_i=k$.
Thus
$$
\sum_{\substack{k_{1},\ldots,k_{s}\\\text{ at least two }k_i>0}}\alpha_Z(p^{k_1},\ldots,p^{k_s})=O(\sum_{k\geq 1}sk^{s-1}p^{-k-1})=O_s(p^{-2}).
$$

Putting these three terms together and letting $m$ tend to
infinity, we get
\begin{multline}
\beta_p=\prod_{j=t+1}^{t+s}\left(1-\chi_{D_j}(p)p^{-1}\right)
\left(\E_{a\in [p]^d}\prod_{i=1}^{t}\Lambda_p(\psi_i(a))\right.\\
+\left.\sum_{j=t+1}^{t+s}
\E_{a\in [p]^d}p1_{\psi_j(a)=0}\prod_{i=1}^{t}\Lambda_p(\psi_i(a))
\sum_{k=1}^{+\infty}\chi_{D_j}(p^k)p^{-k}\right)+O_{s,t}(p^{-2})
\label{betap}
\end{multline}
Green and Tao \cite[Lemma 1.3]{GT2} proved that $\E_{a\in[p]^d}\prod_{i=1}^t
\Lambda_p(\psi_i(a))=1+O_t(p^{-2})$.
%, as appeared in. To prove it, write this expression as
%$$\left(\frac{p}{p-1}\right)^t\Pr((\prod_{i=1}^t
%\psi_i(a),p)=1)$$
%and notice that the probability, because of the affine independence and by some linear algebra and
%inclusion-exclusion, is $1-t/p+O(p^{-2})$
%for $p$ large enough depending on $d,t,L$.
%while the front factor is $1+t/p+O(p^{-2})$ whence the result.
Similarly, for any $j\in\intnn{t+1}{t+s}$, we have
\begin{align*}
\E_{a\in [p]^d}p1_{p\mid\psi_j(a)}\prod_{i=1}^{t}\Lambda_p(\psi_i(a))
&=p\left(\frac{p}{p-1}\right)^t\Pr((\prod_{i=1}^t
\psi_i(a),p)=1)\text{ and }p\mid\psi_j(a))\\
&=1+O(p^{-2})
\end{align*}
because the probability is $p^{-1}(1-t/p+O(p^{-2}))$ by linear
independence.
Moreover,
$$
\prod_{j=t+1}^{t+s}\left(1-\chi_{D_j}(p)p^{-1}\right)
\left(1+\sum_{j=t+1}^{t+s}\sum_{k_j>0}\chi_{D_j}(p^{k_j})p^{-k_j}\right)=1+O_s(p^{-2})
$$
so that finally, plugging these estimates in \eqref{betap},
we obtain
$$
\beta_p=
\prod_{j=t+1}^{t+s}\left(1-\chi_{D_j}(p)p^{-1}\right)
\left(1+\sum_{j=t+1}^{t+s}\sum_{k_j>0}\chi_{D_j}(p^{k_j})p^{-k_j}+O_{s,t}(p^{-2})\right)=1+O(p^{-2}).
$$
Here the implied constant depends
on $t,d,s,L$ and the discriminants only. This last equation is exactly the claimed result.
\end{proof}

\section{Verification of the linear forms condition} 
This section is dedicated to the lengthy and technical
proof of Proposition \ref{linforms}, that is, the verification
that our majorant, introduced in Subsection \ref{majorant},
satisfies the linear forms condition.
We loosely follow Matthiesen's proof in \cite{Matt},
taking inspiration of the more recent paper \cite{MattBr}. However,
there is some flaw there, as the author overlooked the possibility 
that $u$ and $dm^2\epsilon$ may not be coprime; we provide,
based on the earlier paper \cite{Matt2}, a corrected version of these computations.

Compared to Matthiesens's articles, the introduction of the majorant for the von Mangoldt function
adds factors of $\log R$ which will be cancelled during the Fourier
transformation step. It also adds factors of $\frac{\phi(W)}{W}$ which 
remain untouched throughout the proof. And in the core
of the calculation, it adds to the variables $d,m,e,u$ another variable $\ell$ also ranging among the integers whose prime factors are all greater than $w(N)$, which shall interact nicely with the other ones. The aim of the game is 
to dissociate the factors, that is, to transform the average
of the product into the product of averages.

\textbf{Notational conventions for the proof}. In order to somewhat lighten the formidable notation, we will not
always specify the range on sums, products or integrals. In
principle, the name of the variable alone should tell the reader 
what its range is. We list a few important conventions. 
\begin{itemize}
\item The integer vector $n$ will always range in $\Z^d\cap K$. 
\item We put $\phi_j(n)=\overline{W}\psi_j(n)+b_j$, for
$j\in[t+s]$, where $b_j$ lies in the set $B_{t,s}$ defined in Definition~\ref{bts}. 
\item For $i=1,\ldots,t$ and $k=1,2$, $\ell_{i,k}$ is a positive integer. Because it will always be a divisor of $\phi_i(n)$ which satisfies $\phi_i(n)\equiv b_i\mod W$ and $(b_i,W)=1$ by definition of
$B$, the prime factors of $\ell_{i,k}$ are all greater than $w(N)$.
%smaller than $N^{\gamma}$ dividing $\phi_k(n)$
\item For $j=t+1,\ldots,t+s$ and $k=1,2$, $e_{j,k}$ is a positive integer 
%smaller than $N^{\gamma}$ 
in $\langle\Q_{k}\rangle$,  where $\Q_k=\Q_{D_k}$.
All its prime factors are greater than $w(N)$.
%dividing $\phi_j(n)$
\item For $j=t+1,\ldots, t+s$,
$s_j$ will range from $2/\gamma$ to $(\log\log N)^3$ and
$i_j$ from $\log_2s-2$ to $6\log\log\log N$, while $u_j$ ranges in $U(s_j,i_j)$. The $s_j$ should 
not be confused
with $s$, the number of factors of the form $\nu_{\mathrm{Matt},b}$. Notice that $i$ is also the standard
name of the index ranging in $[t]$ but this should not cause any ambiguity.
\item Occasionally we may want to write $e_j$ for $e_{j,1}$ and
$e_j'=e_{j,2}$ ; similarly $\ell_i=\ell_{i,1}$ and $\ell_i'=\ell_{i,2}$.
Moreover $\epsilon_j$ will be the least common multiple (lcm) of $e_j$ and $e_j'$,
while $\lambda_i$ will be the lcm of $\ell_i$ and $\ell_i'$.
\item For $j=t+1,\ldots, t+s$, the integer $d_j$ only has prime factors greater than $w(N)$ and lying in
$\P_{j}$ where  $\P_j=\P_{D_j}$.
%and it will be  smaller than $N^{\gamma}$.
\item For $j=t+1,\ldots, t+s$, the integer $m_j$ only has prime factors greater than $w(N)$ and lying in
$\Q_{j}$.
%and be  smaller than $N^{\gamma}$
\item A bold character denotes a vector; thus $\mathbf{e}=(e_{j,k})_{\substack{j\in\intnn{t+1}{t+s}\\k=1,2}}$ and again
the range of such indices $i,k$ will frequently be omitted.
Unfortunately, the symbol $\ell$ cannot be boldfaced, but this should not cause any ambiguity.
%\item 
%Moreover, we will use the shortcut $\prod\mu$ to denote
%the product of all relevant Möbius factors, generally $\mu(e_i)\mu(e_i')\mu(\ell_j)\mu(\ell_j')$. Similarly
%we will use $\prod\chi$ to denote the product of relevant $\chi$ factors.
\end{itemize}
With these conventions, 
recalling the definitions \eqref{lambdachir}
of $\Lambda_{\chi,R}$ and \eqref{rdgamma} of $r_{D,\gamma}$,
we expand the left-hand side of \eqref{eqlinforms} as
\begin{equation}\label{expanded}
\begin{split}
&\Omega= H
\E_{n\in\Z^d\cap K}\prod_{i\in[t]}
\sum_{\ell_{i},\ell_i'}\mu(\ell_i)\mu(\ell_i')
\chi \left(\frac{\log \ell_i}{\log R}\right)\chi \left(\frac{\log \ell'_i}{\log R}\right)1_{\lambda_i\mid\phi_i(n)}\\
&\prod_{j=t+1}^{t+s}\sum_{s_j,i_j,u_j}2^{s_j}1_{u_j\mid\phi_j(n)}
\sum_{d_j,m_j,e_j,e_j'}1_{d_jm_j^2\epsilon_j\mid \phi_j(n)}\mu(e_j)\mu(e_j')
\chi\left(\frac{\log e_j}{\log R}\right)\chi \left(\frac{\log e_j'}{\log R}\right)\chi \left(\frac{\log d_j}{\log R}\right)
\chi \left(\frac{\log m_j}{\log R}\right).
\end{split}
\end{equation}
The initial factor $H$ is defined by $$H=\left(\log R\frac{\phi(W)}{W}\right)^t\prod_{j=t+1}^{t+s}C_{D_j,\gamma}^{-1}.$$
%So the first term of \eqref{eqlinforms} can be expanded as
%\begin{multline}\label{expanded}
%\E_n\left(\log R\frac{\phi(W)}{W}\right)^{t}
%\sum_{\substack{\ell_{k,j}\\k\in[t],j=1,2}}\prod_{k\in[t],j=1,2}\mu(\ell_{k,j})\chi (\frac{\log \ell_{k,j}}{\log R})1_{\ell_{k,j}\mid\phi_k(n)}\\
%\times\sum_{s_k,u_k,i_k}2^{s_k}\tau(u_k)
%\sum_{\substack{d_k,m_k,e_{k,j}\\k>t,j=1,2}}
%\prod_{k>t,j=1,2}\chi (\frac{\log m_{k}}{\log R})\mu(e_{k,j})\chi (\frac{\log e_{k,j}}{\log R}) 1_{u_kd_km_k^2\epsilon_k\mid\phi_k(n)}
%\end{multline}
Proving Proposition~\ref{linforms}
means proving that
$$
\Omega=1+O_D\left(
\frac{N^{d-1+O_D(\gamma)}}{\Vol (K)}\right)+o_D(1).
$$

Write $\Omega=H\Omega'$.
Notice that $H=\Omega/\Omega'=O((\log R)^t)
=O((\log N)^t)$.
We now work on $\Omega'$. It is an average over $n$ of $t+s$ products, and we aim at transforming
it into a product of $t+s$ averages. We will remember to multiply the error terms obtained
during the transformation of this average by $(\log N)^t$
to obtain error terms for $S$.

We observe that when $u_j,d_j,m_j,e_j,e_j'$ 
divide $\phi_j(n)$ and
$u_j$ satisfies \mbox{$\gcd(u_j,\phi_j(n)/u_j)=1$}, there exists,
for $x$ equal to any of the symbols $e,e',d,m$, a unique decomposition \begin{equation}
\label{vjx}
x_j=\tilde{x_j}v_{j,x}
\quad \text{with}\quad\gcd(\tilde{x_j},u_j)=1\quad \text{and} \quad v_{j,x}\mid u_j.
\end{equation}
We would very much like to perform this decomposition, but not every term satisfies the required coprimality condition.
However, the following claim shows that we can pretend it does at a small cost.

\noindent\textbf{Claim 1.}
The summands in \eqref{expanded} satisfying 
\mbox{$\gcd(u_j,\phi_j(n)/u_j)>1$} for some $j$
or \mbox{$\gcd(u_j,u_i)>1$} for some $i\neq j$ contribute only $O(N^{-(\log\log N)^{-3}/8})$.
%+O(N^{d-1+3\gamma}/\Vol(K)}$.

\begin{proof}
We write the contribution of these summands as
$$S=H\sum_{\mathbf{i},\mathbf{s}}\prod_{j=t+1}^{t+s}2^{s_j}
\E_na_n,$$
where
\begin{align*}a_n=a_{n,\mathbf{i},\mathbf{s}}
=\sum_{\mathbf{u}}1_{\substack{\exists j\mid \gcd(u_j,\phi_j(n)/u_j)>1\\
\text{or }\exists i\neq j\mid \gcd(u_j,u_i)>1}}
&\sum_{\mathbf{d},\mathbf{m},\mathbf{e},\mathbf{\ell}}
\prod_{i=1}^t\mu(\ell_i)\mu(\ell_i')
\chi \left(\frac{\log \ell_i}{\log R}\right)\chi \left(\frac{\log \ell'_i}{\log R}\right)1_{\lambda_i\mid\phi_i(n)}\\
&\prod_{j=t+1}^{t+s}\mu(d_j)\mu(\ell_j)\mu(\ell_j')\chi\left(\frac{\log e_i}{\log R}\right)\chi\left(\frac{\log e_i'}{\log R}\right)1_{\Delta_j\mid\phi_j(n)}
\end{align*}
with the notation $\Delta_j=\gcd(u_j,d_jm_j^2e_j)$.
To bound $\E_n a_n$, we
apply the simple rule, based on Cauchy-Schwarz, that
$$
\left(\E_{n\in\Z^d\cap K}a_n\right)^2\leq \Pr_n(a_n\neq 0)\E_n a_n^2.
$$
%We remark that for any nonzero summand in $S$, there exists a prime $p\geq N^{1/(\log\log N)^3}$ satisfying
%$p^2\mid\prod_j\phi_j(n)$.
%We then basically apply the simple rule based on Cauchy-Schwarz
%$$
%\left(\E_{n\in\Z^d\cap K}a_n\right)^2\leq \Pr(a_n\neq 0)\E a_n^2.
%$$
Now if $a_n\neq 0$ then either the value of one of the $s$ 
last linear forms $\phi_i(n)$ has a repeated prime factor, or 
the values of
two of the $s$ last linear forms have a common prime factor. Such a prime $p$ is a factor of some
$u_i$, which, by Definition~\ref{setsU}, only has prime factors larger than $N^{1/(\log\log N)^3}$ and satisfies $u_i\leq N^{\gamma}$ (see
\cite[Proposition 4.2]{Matt2}). Thus $p$ 
certainly lies between $N^{1/(\log\log N)^3}$ and $N^{\gamma}$.
Using the triangle inequality, we get
$$\Pr_n(a_n\neq 0)\leq \sum_{N^{1/(\log\log N)^3}\leq p\leq N^{\gamma}}\Pr_n(p^2\mid\prod_{i=1}^{t+s}\phi_i(n)).$$
Moreover, the primes $p$ in this range are not exceptional primes, i.e. primes modulo which the linear forms are affinely dependent. Indeed, such primes, thanks to the $W$-trick and the fact that no two of the original linear forms $\psi_i$ are affinely dependent, are all $O(w(N))=O(\log\log N)$. 
Thus $$\Pr_n(p^2\mid\prod_i\phi_i(n))\ll p^{-2}+O\left(\frac{N'^{d-1}}{\Vol(K)}\right)=p^{-2}+O\left(p^2\frac{N^{d-1}}{\Vol(K)}\right),$$
according to Proposition \ref{pcarre}
and the fact that $\abs{K\cap \Z^d}\sim\Vol(K)$.\footnote{Here, we assume that
$\Vol(K)\gg N'^{d}$ or at least that $N'^{d-1}=o(\Vol(K))$. Indeed, in the statement of the main theorem, we could also
add the assumption that $\Vol(K)\gg N^d$ because otherwise the error term is not smaller than the main term.} Hence
\begin{align*}
\Pr(a_n\neq 0)&\leq \sum_{N^{1/(\log\log N)^3}\leq p\leq N^{\gamma}}\Pr(p^2\mid\prod_i\phi_i(n))\\
&\ll\sum_{p\geq N^{1/(\log\log N)^3}}p^{-2}+\frac{N^{d-1}}{\Vol(K)}\sum_{p\leq N^{\gamma}}p^2\\
&\ll N^{-1/(\log\log N)^3}
+\frac{N^{3\gamma+d-1}}{\Vol(K)}.
\end{align*}
Assuming that $\gamma$ is small enough (less than 1/3), the second term is $O(N^{-c})$ with $c>0$ so it is negligible 
with respect to the first one.

We then bound $\E_na_n^2$
quite crudely as follows 
%noticing that the $\mu$ and
%$\chi$ are 1-bounded and using Hölder's inequality
\begin{equation*}
\begin{split}
\E_na_n^2 &\leq \E_n\left(\sum_{\mathbf{d},\mathbf{m},\mathbf{e},\mathbf{\ell},\mathbf{u}}\prod_{i=1}^t
1_{\lambda_i\mid\phi_i(n)}
\prod_{j=t+1}^{t+s}1_{\Delta_j\mid\phi_j(n)}\right)^2\\
&\ll \prod_{i=1}^t\left(\E_n\left(\sum_{\ell_i,\ell_i'}1_{\lambda_i\mid\phi_i(n)}\right)^{2(t+s)}\right) ^{1/(t+s)}
\prod_{j=t+1}^{t+s}\left(\E_n\left(\sum_{d_j,m_j,e_j,e_j',u_j}1_{\Delta_j\mid\phi_j(n)}\right)^{2(t+s)}
\right)^{1/(t+s)}\\
&\ll (\log N)^{O_{t,s}(1)}.\end{split}\end{equation*}
The first inequality is a consequence of the fact that $\abs{\mu}\leq 1$ and $\abs{\chi}\leq 1$.
The second one is H\"{o}lder's.
The last one follows from 
bounds of Matthiesen \cite[Lemma 3.1]{Matt2} on
moments of the divisor function,
and the observation that for instance
$\sum_{\ell_i,\ell_i'}1_{\lambda_i\mid\phi_i(n)}\leq \tau(\phi(n))^2$.
Thus $\abs{\E_na_n}\ll N^{-(\log\log N)^{-3}/4} $.
Summing now over $\mathbf{i},\mathbf{s}$
and multiplying by $H$, we get 
$\abs{S}\leq N^{-(\log\log N)^{-3}/8} $ as desired. This concludes the proof of Claim 1.
\end{proof}

Thus to evaluate \eqref{expanded}, we shall pretend all summands satisfy the coprimality condition,
transform them under this hypothesis, and then reintegrate 
the formerly excluded terms, which generates an error term of size 
$O(N^{-(\log\log N)^{-3}/8})$. 
So \eqref{expanded} is equal to

\begin{equation}\label{newclone}
\begin{split}
&\sum_{\mathbf{i},\mathbf{s}}\sum'_{\mathbf{u}}
\E_n
\prod_{i\in[t],k=1,2}\sum_{\ell_{i,k}}\mu(\ell_{i,k})\chi\left(\frac{\log \ell_{i,k}}{\log R}\right) 1_{\lambda_i\mid \phi_i(n)}
\prod_{j=t+1}^{t+s}2^{s_j}\sum_{d_j,e_j,e_j',m_j\text{ coprime to }u_j}\\
&\sum_{\substack{v_{j,d},v_{j,m},v_{j,e},v_{j,e'}\\\text{ divisors of }u_j}}
\prod_{x_j\in\{d_j,e_j,e_j',m_j\}}\chi\left(\frac{\log x_jv_{j,x}}{\log R}\right)\mu(e_jv_{j,e})\mu(e_jv_{j,e'}) 
1_{u_jd_j\epsilon_jm_j^2\mid\phi_j(n)}
+O(N^{-(\log\log N)^{-3}/8})
%+O(N^{d-1+3\gamma}/\Vol(K)}
\end{split}
\end{equation}
where the dashed sum indicates a sum over vectors whose entries are coprime.

Above and from now on, the vectors $\mathbf{d,e,m}$ are assumed to be entrywise
coprime to the vector $\mathbf{u}$.
Thus we can perform the decomposition \eqref{vjx}.
The vector $\mathbf{v}$ stands for
$(v_{j,x})_{x\in\{d,e,e',m\},j\in\intn{t+1}{t+s}}$ where we impose for every $j$ the conditions
$v_{j,x}\mid u_j$ and $v_{j,d}\in \left\langle \P_j\right\rangle,v_{j,m}\in \left\langle \Q_j\right\rangle,v_{j,e}\in \left\langle \P_j\right\rangle$.
Furthermore, we shall use the notation
$$
q_j=\left\lbrace \begin{array}{cc}
\lambda_j & \text{ if } j\in[t]\\ 
u_jd_j\epsilon_jm_j^2 & \text{ if }
j\in\intnn{t+1}{t+s}.
\end{array} \right.
$$

\noindent\textbf{Claim 2.}
The main term of \eqref{newclone} is equal to
\begin{equation}\label{newcltwo}
\begin{split}
&\sum_{\mathbf{i},\mathbf{s}}\sum_{\mathbf{u}}
\sum_{\mathbf{d,e,m,\ell}}
\alpha(q_1,\ldots,q_{t+s})\\
&\sum_{\mathbf{v}}
\prod_{i\in[t],k=1,2}\mu(\ell_{i,k})\chi\left(\frac{\log \ell_{i,k}}{\log R}\right)
\prod_{j\in\intn{t+1}{t+s}}\frac{2^{s_j}}{u_j}\mu(e_jv_{j,e})\mu(e_jv_{j,e'}) 
\prod_{x_j\in\{d_j,e_j,e_j',m_j\}}\chi\left(\frac{\log x_jv_{j,x}}{\log R}\right)
\end{split}
\end{equation}
up to an error of size $O\left(N^{d-1+O(\gamma)}/\Vol(K)\right)$.

We remark that this error term, after multiplication by the initial factor $H=O((\log N)^t)$, is still of the same magnitude.
\begin{proof}
First, we apply Lemma \ref{lm:localdens}
$$
\E_{n\in\Z^d\cap K}\prod_{i=1}^{t}
1_{\lambda_i\mid\phi_i(n)}\prod_{j=t+1}^{t+s}
1_{u_jd_jm_j^2\epsilon_j\mid\phi_j(n)}=\alpha(q_1,\ldots,
q_{t+s})+O(N^{d-1+O(\gamma)}/\Vol (K)).
$$
To explain the error term, observe that
for any set of tuples bringing a nonzero contribution, 
for any $j\in [t+s]$, we have  $q_j=N^{O(\gamma)}$
because $d_j,m_j,e_j,e_j'\leq N^{\gamma}$  and $u_j\leq N^{\gamma}$. 
%because of the definition
%of $U(i_j,s_j)$ (cf the remark (3) after the Lemma 4.2 of \cite{Matt}). 
%\begin{align*}
%&\E_{n}\prod_{j=t+1}^{t+s}1_{u_jd_jm_j^2\epsilon_j\mid\phi_j(n)}\prod_{i\in[t]}
%1_{\lambda_i\mid \phi_i(n)}\\
%&=\alpha((\lambda_i)_{i\in[t]},(u_jd_jm_j^2\epsilon_j)_{j=t+1,\ldots, t+s}) +O(N^{d-1+O(\gamma)}/\Vol (K)).
%\end{align*}
To bound the contribution of this error term to the sum defining the main term of \eqref{newclone},
we simply notice that the number of terms is $O(N^{\gamma})$ anyway, that the $\mu$ and $\chi$ factors
are 1-bounded, and that $2^{s_j}$ is always $o(N^{\gamma})$ because $s\leq (\log\log N)^3$.

Notice that we can also exclude summands for which 
$\gcd(\lambda_i,u_j)>1$ for some
$i\in[t]$ and $j\in\intn{t+1}{t+s}$ because of the very same argument as in Claim 1. For summands satisfying to
the contrary $\gcd(\lambda_i,u_j)=1$, by multiplicativity of
$\alpha$ and because of the other implicit coprimality conditions, we can write
$$\alpha(q_1,\ldots,q_{t+s})
=\frac{\alpha((\lambda_i)_{i\in[t]},(d_jm_j^2\epsilon_j)_{j\in\intn{t+1}{t+s}})}{\prod_ju_j}
.$$
This concludes the proof of this claim with a dashed sum on $u$ instead of the normal sum,
and a sum on  $\mathbf{\ell}$ restricted to tuples satisfying $\gcd(\lambda_i,u_j)>1$ for all $i$ and $j$.
We can reintegrate now the formerly excluded terms
because they have a negligible contribution anyway, so Claim 2 is proven.
\end{proof}
Having removed the terms $u_j$ from the local density 
$\alpha$, we reset the definition of $q_j$ as follows
$$
q_j=\left\lbrace \begin{array}{cc}
\lambda_j & \text{ if } j\in[t]\\ 
d_j\epsilon_jm_j^2 & \text{ if }
j\in\intnn{t+1}{t+s}.
\end{array} \right.
$$
From now on, we fix vectors $\mathbf{i,s}$ in their usual ranges, and consider the individual terms

\begin{equation}\label{beforechi}
\begin{split}
&\sum_{\mathbf{u}}
\sum_{\mathbf{d,e,m,\ell}}
\alpha(q_1,\ldots,q_{t+s})\prod_{i\in[t],k=1,2}\mu(\ell_{i,k})\chi\left(\frac{\log \ell_{i,k}}{\log R}\right)\\
&\sum_{\mathbf{v}}
\prod_{j\in\intn{t+1}{t+s}}\frac{2^{s_j}}{u_j}\mu(e_jv_{j,e})\mu(e_j'v_{j,e'}) 
\prod_{x_j\in\{d_j,e_j,e_j',m_j\}}\chi\left(\frac{\log x_jv_{j,x}}{\log R}\right)
\end{split}
\end{equation}
We now use the Fourier transform. 
Letting $\theta$ be the Fourier transform of the smooth compactly supported function $x\mapsto e^x\chi (x)$,
it is well known that  
\begin{equation}\label{fourier}
\forall A>0
\qquad\theta (\xi)\ll_A (1+\abs{\xi})^{-A}.
\end{equation}
This allows us to reconstruct $\chi$ from $\theta$ as
an integral over the compact interval\footnote{We prefer integrating over a compact set, in order to 
be able to easily swap summation and integration using
Fubini's theorem.}
$$I=\{\xi\in\R\mid
\abs{\xi}\leq \log^{1/2} R\}$$ 
at the cost of a tolerable error; more precisely, for
any $A>0$, we have

\begin{equation}
\label{chiintegral}
\begin{split}
\chi\left(\frac{\log x}{\log R}\right) &=\int_{\R} x^{-\frac{1+i\xi}{\log R}}\theta(\xi)d\xi\\
&=\int_Ix^{-\frac{1+i\xi}{\log R}}\theta(\xi)d\xi + O(x^{-\frac{1}{\log R}}\log^{-A} R).
\end{split}
\end{equation}

When plugging this into our sum, we need $4s+2t$ real variables
$\xi_{j,k}$ with $k=1,\ldots,4$ for $j=t+1,\ldots, t+s$ and $k=1,2$ for $j=1,\ldots,t$. 
Collectively, they form the vector $\mathbf{\Xi}$.
Furthermore, we write $z_{j,k}=(1+i\xi_{j,k})/(\log R)$. 
We sometimes allow, for a function $f$, 
the slight abuse of notation
$$\prod_{j,k}f(\xi_{j,k})=\prod_{i\in[t],k\in[2]}f(\xi_{i,k})
\prod_{j\in[t+s]\setminus[t],k\in[4]}f(\xi_{j,k}),$$
and write
$$
\theta(\mathbf{\Xi})=\prod_{j,k}\theta(\xi_{j,k}).
$$
We introduce the notation
$\tilde{x}_j=x_jv_{j,x}$ for $x$ equal to any of the symbols $e,e',d,m$,
and $\mathbf{v_i}=(v_{i,d},v_{i,e},v_{i,e'},v_{i,m})$.
For any fixed values of the tuples
$\mathbf{s,i,u,v,d,m,e,\ell}$ we write
$$
M=\prod_{i\in[t],k=1,2}\mu(\ell_{i,k})\prod_{j\in\intn{t+1}{t+s}}\frac{2^{s_j}}{u_j}\mu(e_jv_{j,e})\mu(e_j'v_{j,e'}). 
$$
Finally, we introduce
\begin{equation}
\label{FXi}
F_{\mathbf{d,m,e,\ell}}(\mathbf{\Xi})=F(\mathbf{\Xi})=\theta(\mathbf{\Xi})\prod_{j>t}\tilde{e}_{j,1}^{-z_{j,1}}\tilde{e}_{j,2}^{-z_{j,2}}\tilde{d}_j^{-z_{j,3}}
\tilde{m}_j^{-z_{j,4}}
\prod_{i\in[t]} \ell_{i,1}^{-z_{i,1}}\ell_{i,2}^{-z_{i,2}}.
\end{equation}
We now insert \eqref{chiintegral} into the expression \eqref{beforechi} to get
\begin{equation*}
%\label{beerk}
\sum_{\mathbf{d},\mathbf{m},\mathbf{e},\mathbf{\ell}}\alpha(q_1,\ldots,q_{t+s})\sum_{\mathbf{u,v}}M
\Big(
\int_{I^{4s+2t}}F(\mathbf{\Xi})d\mathbf{\Xi}+
O((\log R)^{-A}(\prod_{i,j,k}\tilde{e}_{j,k}\ell_{i,k}\tilde{d_j}\tilde{m}_j)^{-1/\log R})\Big).
\end{equation*}

Here the term arising from the big oh will not matter too much,
thanks to the following claim.

\noindent
\textbf{Claim 3.} For $A>0$ large enough,
$$
H\sum_{\mathbf{s},\mathbf{i}}\sum_{\mathbf{u,v}} \sum_{\mathbf{d,m,e,\ell}}\alpha(q_1,\ldots,q_{t+s})M
\log^{-A}R\left(\prod_{i\in[t],k=1,2}\ell_{i,k}\prod_{j>t}
\tilde{e}_{j,k}\tilde{m}_j\tilde{d}_j\right)^{-1/\log R}=o(1).
$$

\begin{proof}
Matthiesen \cite[Proposition 4.2]{Matt2} showed that
$$
%\sum_{\mathbf{s},\mathbf{i}}\prod_{j=t+1}^{t+s}\sum_{u_j\in U(s_j,i_j),\mathbf{v_j}}\frac{2^{s_i}}{u_j}
%\ll
\sum_{\mathbf{s},\mathbf{i}}\prod_{j=t+1}^{t+s}\sum_{u_j\in U(s_j,i_j)}\frac{2^{s_j}}{u_j}=O(1).
$$

On the other hand, we can suppress the sum over $\mathbf{v}$ by reintegrating into the sum over $\mathbf{d,m,e}$
the summands not termwise coprime to $\mathbf{u}$.
We can then drop the $\tilde{\cdot}$ on the variables.
We put $q_j'=\ell_j\ell_j'$ for $j\in[t]$ and $q_j=e_je_j'd_jm_j$ for $j\in [t+s]\setminus [t]$.
By multiplicativity,
\begin{equation*}
\begin{split}
\sum_{\mathbf{d},\mathbf{m},\mathbf{e},\mathbf{\ell}}\alpha(q_1,\ldots,q_{t+s})
\left(\prod_{\substack{i\in[t]\\k=1,2}}\ell_{i,k}\prod_{j>t}
e_{j,k}m_jd_j\right)^{-\frac{1}{\log R}}
&=\sum_{\mathbf{d,m,e,\ell}}\prod_{p^{a_i}\Vert q_i}\alpha(p^{a_1},\ldots,p^{a_{t+s}})\prod_{\substack{j\in[t+s]\\p^{a_j'}\Vert q_j'}}p^{-\frac{a_j'}{\log R}}\\
&\leq \sum_{\mathbf{d,m,e,\ell}}\prod_{p^{a_i}\Vert q_i}p^{-\max a_i(1+2\log^{-1}R)}\\
&\leq\prod_{p}(1-p^{-(1+2\log^{-1}R)})^{-O(t+s)}\\
&\ll \log^{O(t+s)}N.
%\\
%&=\sum_{\mathbf{d},\mathbf{m},\mathbf{e},\mathbf{\ell}}
%\prod_{\substack{p>w(N)\\p^{a_i\mid v_i}}\alpha(p^{a_1},\ldots,
%p^{a_{t+s}})\prod_{
\end{split}
\end{equation*}
Here we used $a_j'\geq a_j/2$, Corollary~\ref{cor:localdens} and a crude bound $k^{O(t+s)}$ for the number of tuples $a_i$ satisfying $\max_ia_i=k$.
The last inequality follows from a well-known estimate for
the Zeta function near 1, namely
$$\zeta(x)=O\left(\frac{1}{x-1}\right).$$
Given that $H=O(\log^tN)$, the claim follows for $A$ large
enough depending on $t$ and $s$ only.
\end{proof}

We are left to deal with
\begin{equation}
\label{beeeerk}
\sum_{\mathbf{d},\mathbf{m},\mathbf{e},\mathbf{\ell}}\alpha(q_1,\ldots,q_{t+s})\sum_{\mathbf{u,v}}M
\int_{I^{4s+2t}}F(\mathbf{\Xi})d\mathbf{\Xi}.
\end{equation}
%Thus when choosing $A>0$ large enough we find that
%\begin{footnotesize}
%$$
%\left(\log R\frac{\phi(W)}{W}\right)^t\sum_{\mathbf{s},\mathbf{i},\mathbf{d},\mathbf{m},\mathbf{e},\mathbf{\ell}}\left(\alpha((\lambda_i)_{i\in[t]},(d_im_i^2\epsilon_i)_{i>t})O(\log^{-A}N^{\gamma}(\prod_{j,k}e_{j,k}l_{j,k}d_jm_j)^{-1/\log R})\right)=\log^{-A/2}N^{\gamma}
%$$
%\end{footnotesize}
%which is $o(1)$.
%So incurring only an error $o(1)$, which is tolerable for
%\eqref{eqlinforms}, we can replace \eqref{beforechi}
%by
%\begin{multline}
%\sum_{\mathbf{d},\mathbf{m},\mathbf{e},\mathbf{\ell}}\alpha((\lambda_i)_{i\in[t]},(d_jm_j^2\epsilon_j)_{j>t})\prod_{i>t}\sum_{u_i,\mathbf{v_i}} \frac{2^{s_i}}{u_i}\\
%\times
%\int_{I^{4s+2t}}\prod\mu \prod_{j>t}\tilde{e}_{j,1}^{-z_{j,1}}\tilde{e}_{j,2}^{-z_{j,2}}\tilde{d}_j^{-z_{j,3}}\tilde{m}_j^{-z_{j,4}}
%\prod_{j\in[t]} l_{j,1}^{-z_{j,1}}l_{j,2}^{-z_{j,2}}\prod_{k} \theta(\xi_{j,k})d\xi_{j,k}.
%\end{multline}
We now swap the summation
 $\sum_{\mathbf{d},\mathbf{m},\mathbf{e},\mathbf{\ell}}$
and the integration over the compact set $I^{4s+2t}$, using Fubini's theorem.
This causes no problem because the sum is absolutely convergent;
we are not explicitly going to prove the absolute convergence, but
it follows from the bounds we are going to derive in the
proof of the next claim.

We also continue swapping summation and multiplication, by enforcing
at little cost an extra
coprimality condition: we show we can restrict
to tuples where $(d_jm_j\epsilon_j,\lambda_i)=1$ for all $i,j$ and
$(d_im_i\epsilon_i,d_jm_j\epsilon_j)=1$ for all $i\neq j$.
We need another, more subtle argument to impose this coprimality compared to the coprimality condition involving the
variables $u_j$ in Claim 1, because a crucial ingredient of the proof of this claim was that the prime factors involved
were all at least $N^{(\log\log N)^{-3}}$, an assumption we do not have for $d,m,e$.

\noindent\textbf{Claim 4.}
Let $\mathbf{s,i,u,v}$ be fixed vectors of integers satisfying the usual conditions.
Then we have
\begin{align*}
&\sum_{\mathbf{d},\mathbf{m},\mathbf{e},\mathbf{\ell}}\alpha(q_1,\ldots,q_{t+s})F(\mathbf{\Xi})
\prod_{i,j}\mu(\ell_{i,1})\mu(\ell_{i,2})\mu(\tilde{e}_{j,1})\mu(\tilde{e}_{j,2})  \\
&=(1+O(w(N)^{-1/2}))\sum_{\mathbf{d},\mathbf{m},\mathbf{e},\mathbf{\ell}}'\alpha(q_1,\ldots,q_{t+s})F(\mathbf{\Xi})
\prod_{i,j}\mu(\ell_{i,1})\mu(\ell_{i,2})\mu(\tilde{e}_{j,1})\mu(\tilde{e}_{j,2}) ,
\end{align*}
where the dashed sum is retricted to tuples satisfying $(d_jm_j\epsilon_j,\lambda_i)=1$ for all $i,j$ and
$(d_im_i\epsilon_i,d_jm_j\epsilon_j)=1$ for all $i\neq j$.
%\begin{rmk}

%\end{rmk}

\begin{proof}
%[Proof of Claim 3]
The goal is to bound the contribution of the entries failing
the coprimality conditions.
%The sum to compute may be seen as a (finite) sum over general $t+s$-tuples of integers $(v_1,\ldots,v_{t+s})$, these tuples being then separed into those where the $v_i$ are coprimes
%and the other ones:
%\begin{align*}
%\sum_{(v_1,\ldots,v_{t+s})\in\Z^{t+s}}\alpha(v_1,\ldots,v_{t+s})
%c_{v_1,\ldots,v_{t+s}}&=
%\sum_{(v_1,\ldots,v_{t+s})\in\Z^{t+s}\text{ coprime}}
%\alpha(v_1,\ldots,v_{t+s})
%c_{v_1,\ldots,v_{t+s}}\\
%&+\sum_{(v_1,\ldots,v_{t+s})\in\Z^{t+s}\text{ not coprime}}
%\alpha(v_1,\ldots,v_{t+s})
%c_{v_1,\ldots,v_{t+s}}
%\end{align*}
To achieve this, we exploit the multiplicativity of each summand $T(\mathbf{d,m,e,\ell})$,
which is a term of the form 
\begin{equation}\label{summand}
\alpha(q_1,\ldots,q_{t+s})
\prod_{i=1}^t \ell_{i,1}^{-z_{i,1}}\ell_{i,2}^{-z_{i,2}}\mu(\ell_{i,1})\mu(\ell_{i,2})\prod_{j=t+1}^{t+s}\tilde{e}_{j,1}^{-z_{j,1}}\tilde{e}_{j,2}^{-z_{j,2}}\tilde{d}_j^{-z_{j,3}}\tilde{m}_j^{-z_{j,4}}
\mu(\tilde{e}_{j,1})\mu(\tilde{e}_{j,2}),
\end{equation}
in order to write it as a product over primes; only primes greater than $w(N)$ need be
considered, as smaller ones have no chance of dividing any of the
parameters. We can even
partition the primes $p$ into two classes, according to
whether $p$ divides a single $q_j$ or at least two of them.
Thus, the summand $T$ can be written as
$$
\prod_{p\text{ dividing a single } q_j}\alpha((p^{v_p(q_j)})_j)
A_{p}
\prod_{p\text{ dividing at least two } q_j}\alpha((p^{v_p(q_j)})_j)B_p
$$
where $A_p$ and $B_p$ are complex numbers of modulus at most one and $v_p$ is the $p$-adic valuation.
For any given tuples $\mathbf{d,m,e,\ell}$ and $j\in[s+t]$, we write
$\kappa_j=\prod_{p\text{ dividing at least two }q_j}p^{v_p(q_j)}$. Thus
$$p\mid \kappa_i\Rightarrow p\mid\prod_{j\neq i}\kappa_j.$$
We now arrange the summands $T(\mathbf{d,m,e,\ell})$ according to their tuples $(\kappa_1,\ldots,\kappa_{t+s})$.
Let us fix such a tuple $(\kappa_1,\ldots,\kappa_{t+s})$.
The sum of summands $T$ corresponding to this tuple is

% Thus,
%a typical summand is factored into 
%\begin{enumerate}
%\item a factor arising from the primes dividing at least two
%of the $v_i$, i.e. a factor of the form
%$\alpha(k_1,\ldots,k_{t+s})$ times a complex factor of modulus at most one (this is because $z_{j,k}$ has a positive real part)
%where for any $i$, $p\mid k_i\Rightarrow p\mid\prod_{j\neq i}k_j$.
%\item a factor arising from the primes dividing a single $v_i$,
%which is again of the form of \eqref{summand}
%but with parameters satisfying the coprimality condition.
%\end{enumerate}

%\begin{footnotesize}
\begin{align*}
&\alpha(\kappa_1,\ldots,\kappa_{t+s})\abs{\sum_{
\substack{\mathbf{d},\mathbf{m},\mathbf{e},\mathbf{\ell}\\\forall j\,(q_j,\kappa_j)=1}}'\alpha(q_1,\ldots,q_{t+s})\prod_{\substack{j=t+1\\k=1,2}}^{t+s}
\tilde{e}_{j,k}^{-z_{j,k}}\mu(\tilde{e}_{j,k})\tilde{d}_j^{-z_{j,3}}\tilde{m}_j^{-z_{j,4}}
\prod_{i\in[t]} \ell_{i,k}^{-z_{i,k}}\mu(\ell_{i,k})}\\
&\leq\alpha((\kappa_j))\prod_{p\mid\prod_j\kappa_j}(1+O(p^{-1}))
\abs{\sum_{\mathbf{d},\mathbf{m},\mathbf{e},\mathbf{\ell}}'\alpha((q_j))\prod_{\substack{j=t+1\\k=1,2}}^{t+s}
\tilde{e}_{j,k}^{-z_{j,k}}\mu(\tilde{e}_{j,k})\tilde{d}_j^{-z_{j,3}}\tilde{m}_j^{-z_{j,4}}
\prod_{i\in[t]} \ell_{i,k}^{-z_{i,k}}\mu(\ell_{i,k})}\\
\end{align*}
%\end{footnotesize}
by Möbius inversion (or inclusion-exclusion) and the triangle inequality.
%\begin{align*}
%&\alpha(v_1,\ldots,v_{t+s})\prod_{i=1}^t \ell_{i,1}^{-z_{i,1}}\ell_{i,2}^{-z_{i,2}}\mu(\ell_{i,1})\mu(\ell_{i,2})\prod_{j=t+1}^{t+s}e_{j,1}^{-z_{j,1}}e_{j,2}^{-z_{j,2}}d_j^{-z_{j,3}}m_j^{-z_{j,4}}
%\mu(e_{i,1})\mu(e_{i,2})\\
%&=\prod_{\substack{p>w(N)\\p\text{ divides a single } v_i}}
%\alpha(p^{v_p(v_1)},\ldots, p^{v_p(v_{t+s})})
%\end{align*}
Next we evaluate the sum of all terms of the form
$\alpha(\kappa_1,\ldots,\kappa_{t+s})\prod_{p\mid\prod_j\kappa_j}(1+O(p^{-1}))$. Using multiplicativity,
we can write them as products over primes, all these primes
being larger than $w(N)$.
% because $\phi_j(n)\equiv b_j\mod\overline{W}$ is coprime to $W$.
Thus, a crude bound is
$$
\prod_{p>w(N)}\left(1+\sum_{\substack{a_1,\ldots,a_{t+s}\\
\text{ at least two }a_i>0}}O(a_1^2+\cdots+a_t^2+a_{t+1}^4+
\cdots+a_{t+s}^4)\alpha((p^{a_i}))(1+O(p^{-1}))\right)-1
$$
where we have used the simple bound $\tau_k(p^{a_i})\ll a_i^{k-1}$ for
$k=3$ (because of $\lambda_i=\lambda_i/\ell_i\cdot\lambda_i/\ell_i'
\cdot \ell_i\ell_i'/\lambda_i$, hence the number of occurrences of
$\lambda_i$ is bounded by the number of decompositions of it
into three factors) and for $k=5$ 
(because of $d_jm_j^2\epsilon_j=d_j\cdot m_j^2\cdot \epsilon_j/e_j\cdot \epsilon_j/e_j'\cdot e_je_j'/\epsilon_j$).
The requirement that 
at least two $a_i$ be positive comes from the very definition
of $\kappa_i$.
Notice that the $-1$ is here to remove the 1 arising from
$\alpha(1,\ldots, 1)$.
To further bound this expression,
we first bound $a_i^2$ by $a_i^4$ and recall
that the number of tuples $(a_1,\ldots,a_{t+s})$
satisfying $\max a_i=k$ is at most $t'(k+1)^{t'-1}$ (with $t'=t+s$).
For such tuples, we have $\sum_i a_i^4\leq t'k^4$
and since the system is of finite complexity and at least
two $a_i$ are nonzero, $\alpha((p^{a_i})_{i\in[t+s]})\leq p^{-k-1}$ according to Proposition~\ref{prop:locbounds}.
Thus
$$\sum_{\substack{a_1,\ldots,a_{t+s}\\
\text{ at least two }a_i>0}}O(a_1^2+\cdots+a_t^2+a_{t+1}^4+
\cdots+a_{t+s}^4)\alpha((p^{a_i}))(1+O(p^{-1}))$$
is bounded by
$$\sum_{k\geq1}
p^{-k-1}t'^2k^{t'+3}
\ll\sum_{k\geq 1}p^{-3k/4-1}
\ll p^{-3/2}
$$
the first inequality being provided by obvious growth comparisons valid for large $p$ (we may assume $N$ to be large enough
for $p>w(N)$ to satisfy automatically this condition).
Since
$$
\prod_{p>w(N)}(1+p^{-3/2})-1\leq \sum_{n>w(N)}n^{-3/2}\ll w(N)^{-1/2},
$$
Claim 4 follows.
\end{proof}

The extra coprimality condition that Claim 4 allows us to assume
enables us to write $\alpha$ as the product of the
reciprocals of its arguments, resulting in
\begin{align*}
&\sum_{\mathbf{d},\mathbf{m},\mathbf{e},\mathbf{\ell}}'\alpha(q_1,\ldots,q_{t+s})\prod_{j=t+1}^{t+s}\frac{\mu(e_j)\mu(e_j')}{\epsilon_j}
e_{j,1}^{-z_{j,1}}e_{j,2}^{-z_{j,2}}d_j^{-1-z_{j,3}}m_j^{-2-z_{j,4}}
\prod_{i=1}^t\frac{\mu(\ell_i)\mu(\ell'_i)}{\lambda_i}
\ell_{i,1}^{-z_{i,1}}\ell_{i,2}^{-z_{i,2}}\\
&=\sum_{\mathbf{d},\mathbf{m},\mathbf{e},\mathbf{\ell}}'\prod_{j=t+1}^{t+s}\frac{\mu(e_j)\mu(e_j')}{\epsilon_j}
e_{j,1}^{-z_{j,1}}e_{j,2}^{-z_{j,2}}d_j^{-1-z_{j,3}}m_j^{-2-z_{j,4}}
\prod_{i=1}^t\frac{\mu(\ell_i)\mu(\ell'_i)}{\lambda_i}
\ell_{i,1}^{-z_{i,1}}\ell_{i,2}^{-z_{i,2}}.
\end{align*}
Notice that the above is an equation without tildes. We will in the sequel
avoid them, observing that for any fixed $\mathbf{u}$, we have
\begin{align*}
&\sum_{\mathbf{d},\mathbf{m},\mathbf{e},\mathbf{\ell}}'\prod_{j=t+1}^{t+s}\frac{\mu(\tilde{e_j})\mu(\tilde{e_j}')}{\epsilon_j}
\tilde{e}_{j,1}^{-z_{j,1}}\tilde{e}_{j,2}^{-z_{j,2}}\tilde{d}_j^{-1-z_{j,3}}\tilde{m}_j^{-2-z_{j,4}}
\prod_{i=1}^t\frac{\mu(\ell_i)\mu(\ell'_i)}{\lambda_i}
\ell_{i,1}^{-z_{i,1}}\ell_{i,2}^{-z_{i,2}}
\\
&=\sum_{\mathbf{d},\mathbf{m},\mathbf{e},\mathbf{\ell}}'\prod_{j=t+1}^{t+s}\frac{\mu(e_j)\mu(e_j')}{\epsilon_j}
e_{j,1}^{-z_{j,1}}e_{j,2}^{-z_{j,2}}d_j^{-1-z_{j,3}}m_j^{-2-z_{j,4}}
\prod_{i=1}^t\frac{\mu(\ell_i)\mu(\ell'_i)}{\lambda_i}
\ell_{i,1}^{-z_{i,1}}\ell_{i,2}^{-z_{i,2}}\\
&\times
\sum_{\mathbf{v}}\mu(v_{j,e})\mu(v_{j,e'})v_{j,e}^{-z_{j,1}}v_{j,e'}^{-z_{j,2}}v_{j,d}^{-z_{j,3}}v_{j,m}^{-z_{j,4}},
\end{align*}
where the sum over $\mathbf{v}$ is as usual over vectors $(v_{j,x})$ where
$v_{j,x}\mid u_j$ and $v_{j,x}$ satisfies the same condition on its prime factors as $x$ (all in $\P_j$ for
$d$ and $e$, all in $\Q_j$ for $m$).

Next we claim that we can remove the dash on the sum.

\noindent\textbf{Claim 5.} The following equality holds, for any choice of the family $\xi_{j,k}$ in
$I=[-\sqrt{\log R},\sqrt{\log R}]$.
%\begin{footnotesize}
\begin{align*} 
&\sum_{\mathbf{d},\mathbf{m},\mathbf{e},\mathbf{\ell}}'\prod_{j=t+1}^{t+s}\frac{\mu(e_j)\mu(e_j')}{\epsilon_j}
e_{j,1}^{-z_{j,1}}e_{j,2}^{-z_{j,2}}d_j^{-1-z_{j,3}}m_j^{-2-z_{j,4}}
\prod_{i=1}^t\frac{\mu(\ell_i)\mu(\ell'_i)}{\lambda_i}\ell_i^{-z_{i,1}}\ell_i'^{-z_{i,2}}
\\
&=(1+O(w(N))^{-1/2})
\sum_{\mathbf{d},\mathbf{m},\mathbf{e},\mathbf{\ell}}\prod_{j=t+1}^{t+s}
\frac{\mu(e_j)\mu(e_j')}{\epsilon_j}e_{j,1}^{-z_{j,1}}e_{j,2}^{-z_{j,2}}d_j^{-1-z_{j,3}}m_j^{-2-z_{j,4}}\times\prod_{i\in[t]} 
\frac{\mu(\ell_i)\mu(\ell_i')}{\lambda_i}\ell_{i,1}^{-z_{i,1}}\ell_{i,2}^{-z_{i,2}}.
\end{align*}
%\end{footnotesize}
\begin{proof}
The justification is basically the same as for Claim 4, because
the claim simply consists in replacing the dashed sum by a complete sum, at the same small cost.
\end{proof}
Let us introduce for any $i\in[t]$ and $\mathbf{l,\Xi}$ the notation
$$
V_i=V_i(\mathbf{l,\Xi})=\frac{\mu(\ell_i)\mu(\ell_i')}{\lambda_i}\ell_i^{-z_{i,1}}\ell_i'^{-z_{i,2}}
$$
and
$$V(\mathbf{l,\Xi})=\prod_{i\in[t]}V_i(\mathbf{l,\Xi}).$$
Similarly, for any $j\in\intnn{t+s}{t+s}$ and tuples $\mathbf{u,v,d,m,e}$ we define
\begin{align*}
S_j(\mathbf{u,v,\Xi})&= \frac{2^{s_j}}{u_j}\mu(v_{j,e})\mu(v_{j,e'})v_{j,e}^{-z_{j,1}}v_{j,e'}^{-z_{j,2}}v_{j,d}^{-z_{j,3}}v_{j,m}^{-z_{j,4}}\\
T_j(\mathbf{d,m,e,\Xi}) &= \frac{\mu(e_j)\mu(e_j')}{\epsilon_j}e_j^{-z_{j,1}}e_j'^{-z_{j,2}}d_j^{-1-z_{j,3}}
m_j^{-2-z_{j,4}}. 
\end{align*}
Finally we put
\begin{equation*}
S(\mathbf{u,v,\Xi}) =\prod_{j=t+1}^{t+s}S_j\qquad\text{ and
%\footnote{This $T$ has nothing to do with the auxiliary $T$ introduced
%during the proof of Claim 4.} 
}
\qquad
T(\mathbf{d,m,e,\Xi}) = \prod_{j=t+1}^{t+s}T_j.
\end{equation*}
With this notation, one can rewrite \eqref{beeeerk} as 
\begin{equation}
\begin{split}
(1+O(w^{-1/2}))\int_{I^{4s+2t}}\theta(\mathbf{\Xi})
\sum_{\mathbf{u,v}}S(\mathbf{u,v,\Xi})
\sum_{\mathbf{d,m,e}}T(\mathbf{d,m,e,\Xi})
\sum_{\mathbf{\ell}}V(\mathbf{l,\Xi}) d\mathbf{\Xi}.
\end{split}
\label{mainterm5}
\end{equation}
Now we show that the error arising from
the $O(w^{-1/2})$ term in \eqref{mainterm5} is indeed negligible: we must ensure that
\begin{equation}
\label{errorTerm}
w^{-1/2}H\sum_{\mathbf{s,i}}\int_{I^{4s+2t}}
\theta(\mathbf{\Xi})
\sum_{\mathbf{u,v}}S(\mathbf{u,v,\Xi})
\sum_{\mathbf{d,m,e}}T(\mathbf{d,m,e,\Xi})
\sum_{\mathbf{\ell}}V(\mathbf{l,\Xi}) d\mathbf{\Xi}=o(1).
\end{equation}
This is because on the one hand
$$\abs{\sum_{\mathbf{v_j}}\mu(v_{j,e})\mu(v_{j,e'})v_{j,e}^{-z_{j,1}}v_{j,e'}^{-z_{j,2}}v_{j,d}^{-z_{j,3}}v_{j,m}^{-z_{j,4}}}
\leq \tau(u_j)^4$$
and 
$$
\sum_{\mathbf{s,i,u}}\prod_{j=t+1}^{t+s}\frac{2^{s_j}\tau(u_j)^4}{u_j}=O(1)
$$
by similar calculations\footnote{The main ingredients are the easy observation that any $u\in U(i,s)$ has $2^{m_0(i,s)}$ divisors and the bound
$\sum_{u\in U(i,s)}u^{-1}\leq (\sum_{p\in I_i}p^{-1})^{m_0}\ll
(\log 2)^{m_0}$, where $I_i=[N^{2^{-i-1}},N^{2^{-i}}]$.
}
 to the ones of Matthiesen
\cite[Proof of Proposition 4.2]{Matt2}.
And on the other hand, the next claim provides a fitting bound.

\noindent\textbf{Claim 6.} We have
%where the integral is over $I^{4s+2t}$.
\begin{equation}\label{integralproduct}
\int\abs{\theta(\mathbf{\Xi})
\sum_{\mathbf{d,m,e}}T(\mathbf{d,m,e,\Xi})
\sum_{\mathbf{\ell}}V(\mathbf{l,\Xi})}  d\mathbf{\Xi} =O(1/(\log R)^{t}),
\end{equation}
where the integral is over $I^{4s+2t}$.

Given that $H=O(\log R)^{t}$, the bound \eqref{errorTerm} follows from this claim. %(recall that the initial $\phi(W)/W$ is at most 1 anyway)
\begin{proof}
%[Proof of Claim 5]
We first replace the sum over $\ell_i,\ell_i'$,
for any $i\in[t]$,
by a product over primes, using multiplicativity, to get
$$
\sum_{\ell_i,\ell_i'}V_i=\sum_{\ell_i,\ell_i'}\frac{\mu(\ell_i)\mu(\ell_i')}{\lambda_j}\ell_i^{-z_{i,1}}\ell_i'^{-z_{i,2}}=\prod_{s\in \P}(1-s^{-1-z_{i,1}}-s^{-1-z_{i,2}}+s^{-1-z_{i,1}-z_{i,2}}).
$$
Then we notice that for large primes $s$ and complex numbers $z,z'$ of positive real part
$$
1-s^{-1-z}-s^{-1-z'}+s^{-1-z-z'}
=\frac{(1-s^{-1-z})(1-s^{-1-z'})}{1-s^{-1-z-z'}}+O(s^{-2}),
$$
so that
$$
\prod_{s\in \P}(1-s^{-1-z_{j,1}}-s^{-1-z_{j,2}}+s^{-1-z_{j,1}-z_{j,2}})\ll\prod_{s\in \P}\frac{(1-s^{-1-z})(1-s^{-1-z'})}{1-s^{-1-z-z'}}.
$$
Finally we recall 
that
%the classical asymptotic near 1 for 
the $\zeta$ function is defined for $\Re z>1$ by
$$
\zeta(z)=\sum_{n\geq 1}n^{-z}=\prod_p(1-p^{-z})^{-1}
$$
and satisfies
$$
\zeta(z)=\frac{1}{z-1}+O(1)
$$
for values of $z$ near 1.
From this fact, a quick computation yields
$$
\prod_{s\in \P}\frac{(1-s^{-1-z})(1-s^{-1-z'})}{1-s^{-1-z-z'}}
\ll \frac{zz'}{z+z'},
$$
whence the bound
\begin{equation}\label{boundprod}
\prod_{s\in \P}(1-s^{-1-z_{i,1}}-s^{-1-z_{i,2}}+s^{-1-z_{i,1}-z_{i,2}})\ll \frac{z_{i,1}z_{i,2}}{z_{i,1}+z_{i,2}}.
\end{equation}
 for any $i\in[t]$ and $\xi_{i,k}\in I$ (for
$k=1,2$) and the corresponding $z_{i,k}$.
Similarly, for any $j\in \{t+1,\ldots,t+s\}$
\begin{equation}\label{prodprime}
\begin{split}
\sum_{d_j,m_j,e_j,e_j'} \frac{\mu(e_j)\mu(e_j')}{\epsilon_j}e_j^{-z_{j,1}}e_j'^{-z_{j,2}}d_j^{-1-z_{j,3}}
m_j^{-2-z_{j,4}}&=\prod_{q\in \Q_j}(1-q^{-1-z_{j,1}}-q^{-1-z_{j,2}}+q^{-1-z_{j,1}-z_{j,2}})\\
&\prod_{r\in \Q_j}(1-r^{-2-z_{j,4}})^{-1} \prod_{p\in \P_j}(1-p^{-1-z_{j,3}})^{-1}.
\end{split}
%&\prod_{s\in \P}(1-s^{-1-z_{2,1}}-s^{-1-z_{2,2}}+s^{-1-z_{2,1}+z_{2,2}})
\end{equation}
Notice that the product in $r$ is a convergent product, bounded by a constant when $z_{j,4}$ varies in the permitted
range.

%We use the well-known expression for
%$$
%\sum_{s\in \P}s^{-1-z}=\log\frac{1}{z}+O(1).
%$$
%Taking logarithms, we have
%\begin{equation*}
%\begin{split}\log \prod_{s\in \P}(1-s^{-1-z_{j,1}}-s^{-1-z_{j,2}}+s^{-1-z_{j,1}-z_{j,2}})&=\sum_{s\in \P}\log(1-s^{-1-z_{j,1}}-s^{-1-z_{j,2}}+s^{-1-z_{j,1}-z_{j,2}})\\
%&=O(1)+\sum_{s\in \P}(-s^{-1-z_{j,1}}-s^{-1-z_{j,2}}+s^{-1-z_{j,1}-z_{j,2}})
%\end{split}
%\end{equation*}
%because the other terms
%in the power series expansion of the logarithm are summable.
%This implies that 
%$$
%\abs{\prod_{s\in \P}(1-s^{-1-z_{i,1}}-s^{-1-z_{i,2}}+s^{-1-z_{i,1}-z_{i,2}})}\ll 
%\prod_{i=1}^t\abs{z_{i,2}}\abs{z_{i,1}}\abs{z_{i,1}+z_{i,2}}^{-1}
%$$
Given that
$\P_j$ and $\Q_j$ each have density 1/2 among the primes,
we can write 
$$
\sum_{q\in\P_j}q^{-1-z}=\frac{1}{2}\log\frac{1}{z}+O(1)
$$
for $\Re z>0$.\footnote{This amounts to saying that if a set of primes has a natural density, it has a Dirichlet density which is equal to its natural density.}
This provides a bound for the product \eqref{prodprime},
similar to the one in \eqref{boundprod}, namely
\begin{align*}
&\prod_{q\in \Q_j}(1-q^{-1-z_{j,1}}-q^{-1-z_{j,2}}+q^{-1-z_{j,1}-z_{j,2}})
\prod_{r\in \Q_j}(1-r^{-2-z_{j,4}})^{-1} \prod_{p\in \P_j}(1-p^{-1-z_{j,3}})^{-1}\\
&\ll\abs{z_{j,1}}^{1/2}\abs{z_{j,2}}^{1/2}
\abs{z_{j,1}+z_{j,2}}^{-1/2}\abs{z_{j,3}}^{-1/2}
\end{align*}
Recall that  $z_{j,k}=(1+\xi_{j,k})(\log R)^{-1}$, thus
$\abs{z_{j,k}}\leq (1+\abs{\xi_{j,k}})(\log R)^{-1}$ by triangle inequality, and
$
\abs{z_{j,1}+z_{j,2}}^{-1}\leq \log R
$
for any $j\in[t+s]$.
Moreover,
\eqref{fourier}
yields
$$
\theta(\mathbf{\Xi})=O_A\left(\prod_{j,k} (1+\abs{\xi_{j,k}})^{-A}\right).
$$
Multiplying all these bounds, we find that the integrand in \eqref{integralproduct}
is bounded by
\begin{align*}
&\prod_{i=1}^t\abs{z_{i,2}}\abs{z_{i,1}}\abs{z_{i,1}+z_{i,2}}^{-1}
\prod_{j=t+1}^{t+s}\abs{z_{j,1}}^{1/2}\abs{z_{j,2}}^{1/2}
\abs{z_{j,1}+z_{j,2}}^{-1/2}\abs{z_{j,3}}^{-1/2}
\prod_{j,k} (1+\abs{\xi_{j,k}})^{-A}\\
&\ll (\log R)^{-t}
\left(\prod_{i=1}^t(1+\abs{\xi_{i,1}})(1+\abs{\xi_{i,2}})\right)^{1-A}\left(\prod_{j=t+1}^{t+s}(1+\abs{\xi_{j,1}})(1+\abs{\xi_{j,2}})
\right)^{1/2-A}\\
&\ll (\log R)^{-t}\prod_{j,k} (1+\abs{\xi_{j,k}})^{-A/2}
\end{align*}
when $A$ is large enough (for the last step).
This last product is certainly integrable as soon as $A>2$, so the final expression is $O((\log R)^{-t})$ as
claimed.
\end{proof}
We now study the main term of
\eqref{mainterm5}.
We can again swap  summation  and integration using Fubini's theorem.
Using separation of variables, we transform the main term of \eqref{mainterm5} into
\begin{multline}
\label{mainterm6}
%\begin{split}
\sum_{\mathbf{u,v},\mathbf{d},\mathbf{e},\mathbf{m},\mathbf{\ell}}
\prod_{i=1}^t 
\int_{I^2}V_i\theta(\xi_{i,1})\theta(\xi_{i,2})d\xi_{i,1} d\xi_{i,2}
\prod_{j=t+1}^{t+s}
\int_{I^4}S_jT_j\prod_{k\in[4]}\theta(\xi_{j,k})d\xi_{j,k}
%\end{split}
\end{multline}

It is now time to
undo the truncation to $I$ in these integrals, in order to be able to 
collapse them into factors of $\chi$.
The error term arising from the removal of this truncation 
is the same as the one introduced by
the truncation, so it can be subsumed into the $o(1)$ of
\eqref{eqlinforms}.
Thus, up to an error term $E_{\mathbf{i,s}}$ satisfying 
$(\log R)^t\sum_{\mathbf{i},\mathbf{s}} E_{\mathbf{i},\mathbf{s}}=o(1)$, the expression
\eqref{mainterm6} is equal to
\begin{multline}
%\begin{split}
\sum_{\mathbf{u,v},\mathbf{d},\mathbf{e},\mathbf{m},\mathbf{\ell}}
\prod_{i=1}^t \frac{\mu(\ell_{i,1})\mu(\ell_{i,2})}{\lambda_i}
\prod_{k=1,2}
\chi\left(\frac{\log \ell_{i,k}}{\log R}\right)\\
\prod_{j=t+1}^{t+s}\frac{2^{s_j}\tau (u_j)}{u_j}\frac{\mu(e_jv_{j,e})\mu(e_j'v_{j,e'})}{d_jm_j^2\epsilon_j}\chi(\frac{\log d_jv_{j,d}}{\log R})
\chi(\frac{\log m_jv_{j,m}}{\log R})\prod_{k=1,2}
\chi(\frac{\log e_{j,k}v_{j,e_k}}{\log R}).
%\end{split}
\end{multline}
%$$
%\prod_{i\in[t_1]}\sum_{u_i}\frac{2^{s_i}\tau(u_i)}{u_i}\sum_{d_i,m_i,e_i,e_i'}\frac{\mu(e_i)\mu(e_i')}{d_im_i^2\epsilon_i}\prod_{j>t_1}\sum_{\ell_j,\ell_j'}\frac{\mu(\ell_j)\mu(\ell'_j)}{\lambda_j}\prod\chi+o(1)
%$$
%So this is what we get for \eqref{eqlinforms},
%up to error terms of the desired form ($O_D\left(
%\frac{N^{d-1+O_D(\gamma)}}{\Vol (K)}\right)$ in Claims 1 and 2,
%various $o(1)$ throughout the proof):
Interchanging summation and multiplication, we find that,
up to error terms of the desired magnitude ($O_D\left(
\frac{N^{d-1+O_D(\gamma)}}{\Vol (K)}\right)$ in Claims 1 and 2,
various $o(1)$ throughout the proof), $\Omega$ equals

\begin{align*}
\prod_{j=t+1}^{t+s}&C_{D_j,\gamma}^{-1}\sum_{s_j,i_j,u_j,v_j}\sum_{d_j,m_j,e_j,e_j'}\frac{2^{s_j}
}{u_j}\frac{\mu(e_jv_{j,e})\mu(e_j'v_{j,e'})}
{d_jm_j^2\epsilon_j}\prod_{x\in\{d,m,e,e'\}}\chi\left(\frac{\log x_jv_{j,x}}{\log R}\right)\\
&\times\prod_{i\in[t]}\left(\log R\frac{\phi(W)}{W}\sum_{\ell_i,\ell_i'}\frac{\mu(\ell_i)\mu(\ell_i')}{\lambda}\prod_{x\in\{\ell_i,\ell_i'\}}\chi\left(\frac{\log x}{\log R}\right)\right),
\end{align*}
which is a product of $t+s$ factors, independent of the
system of linear forms.
Hence the $j$th factor, for $j\in[t+s]$, is also
the main term of the average of the $j$th pseudorandom majorant for the trivial system $\Phi : \Z\rightarrow \Z,n\mapsto n$.
Now because of the properties of the Green-Tao
and the Matthiesen majorant,
described in subsections \ref{majorantgt} and
\ref{majorantmatt} respectively, these averages are $1+o(1)$, whence the result.

\section{Volume packing arguments and local divisor density}
In this appendix, we shall collect
some frequently used facts concerning the number
of solutions to a system
of linear equations in a convex set of $\R^d$ and in $(\Z/m\Z)^d$.
We first recall a lemma already stated earlier but
particularly relevant here, borrowed from Green and Tao \cite[Appendix A]{GT2}.
\begin{lm}\label{vol}
Let $K\subset [0,N]^d$ be a convex body of $\R^d$. Then
$$\abs{K\cap\Z^d}=\sum_{n\in K\cap Z^d}=\Vol(K)+O_d(N^{d-1}).$$
\end{lm}
%\begin{proof}
%See \cite{GT2}, Appendix A.
%\end{proof}
We recall the definition of the local divisor density and we
mention some useful properties.
\begin{dfn}
For a given system of affine-linear forms
$\Psi=(\psi_1,\ldots,\psi_t):\Z^d\rightarrow\Z^t$,
positive integers $d_1,\ldots,d_t$ of lcm $m$, define the \textit{local divisor density} by
$$
\alpha_{\Psi}(d_1,\ldots,d_t)=\E_{n\in (\Z/m\Z)^d}\prod_{i=1}^t
1_{\psi_i(n)\equiv 0\mod d_i}.
$$
\end{dfn}
The following lemma is borrowed from Matthiesen \cite[Lemma 9.3]{Matt}.
\begin{lm}
\label{lm:localdens}
Let $K\subset [-B,B]^d$ be a convex body and $\Psi$ a system of affine-linear forms, and let 
$d_1,\ldots,d_t$ be integers of lcm $m$. Then
$$
\sum_{n\in\Z^d\cap K}\prod 1_{d_i\mid\psi_i(n)}=\Vol (K)
\alpha_{\Psi}(d_1,\ldots,d_t)+O(B^{d-1}m).
$$
\end{lm}
We shall try to bound $\alpha_{\Psi}(p^{a_1},\ldots,p^{a_t})$.
To this aim, we state a version of Hensel's lemma in several
variables. 
\begin{lm}
\label{Hensel}
Let $Q\in\Z[X_1,\ldots,X_d]$, $p$ be a prime and $k\geq 1$
an integer
and \mbox{$x\in (\Z/p^k\Z)^d$} such that  $Q(x)\equiv 0 \mod p^k$ and
$$\grad Q(x)=\left(\frac{\partial Q}{\partial x_1},\ldots, 
\frac{\partial Q}{\partial x_d}\right)(x)\neq 0\mod p.$$ Then there exist precisely
$p^{d-1}$
vectors $y\in (\Z/p^{k+1}\Z)^d$ such that $x\equiv y\mod p^k$
and $Q(y)\equiv 0\mod p^{k+1}$.
\end{lm} 
\begin{proof}
Let $y\in (\Z/p^{k+1}\Z)^d$ satisfy $x\equiv y\mod p^k$; in other words, $y=x+p^kz$ for some uniquely determined $z\in(\Z/p\Z)^d$.
Here we treat $x\in(\Z/p^k\Z)^d$ as an element of 
$(\Z/p^{k+1}\Z)^d$ by using the canonical injection.
We then treat $Q(x)$ as an element
of $\Z/p^{k+1}\Z$ congruent to 0 mod $p^k$ and put $Q(x)=p^ka$ with $a\in \Z/p\Z$.
Then Taylor's formula ensures that 
$$Q(y)\equiv Q(x)+p^k\grad Q(x)\cdot z
\equiv p^k(a+\grad Q(x)\cdot z)\mod p^{k+1}.$$
So $Q(y)\equiv 0 \mod p^{k+1}$ is equivalent
to $a+\grad Q(x)\cdot z\equiv 0\mod p$.
As $\grad Q(x)$ is not zero modulo $p$, this imposes a nontrivial affine equation on $z$ in the vector space $\F_p^d$,
so $z$ is constrained to lie in a $(d-1)$-dimensional
affine $\F_p$-subspace, which has $p^{d-1}$ elements, hence
the conclusion.
\end{proof}
As an application, we prove the following statement.
\begin{cor}\label{cor:localdens}
Let $\psi$ be an affine-linear form in $d$ variables, and let 
$p$ be a prime such that $\psi$ is not the trivial form modulo $p$. Then for any $m\geq 1$
$$
\alpha_m=\alpha_{\psi}(p^m)=
\E_{n\in(\Z/p^m\Z)^d}1_{p^m\mid\psi(n)}=\Pr_{n\in(\Z/p^m\Z)^d}(\psi(n)=0)\leq p^{-m}.
$$
\end{cor}
\begin{rmk}
\label{rmk:locdens}
If $\psi=p^k\psi'$ and $\psi'$ is not the trivial form modulo $p$, this corollary provides for $m\geq k$ the bound
$\alpha_\psi(p^m)\leq p^{k-m}\ll p^{-m}$.
\end{rmk}
\begin{proof}
If $n\in\Zp{m}{d}$ satisfies $\psi(n)\equiv 0\mod p^m$, then in particular
$\tilde{\psi}(\tilde{n})\equiv 0\mod p$,
where $\tilde{\cdot}$ is the reduction modulo $p$,
which imposes that $\tilde{n}$ lies in $\ker\tilde{\psi}$
%(our use of the notation $\ker$ may be slightly improper, because $\psi$ is an affine form, and not a linear one,
%on the space $\F_p^d$)
.
By assumption, $\tilde{\psi}\neq 0$. If
its linear part is 0, then its constant part is nonzero,
thus $\ker\tilde{\psi}=\emptyset$ and $\alpha_m=0$. Otherwise, the linear part is nonzero
modulo $p$, and then $\ker\tilde{\psi}$ is an affine $\F_p$-hyperplane, thus has $p^{d-1}$ elements. 
Let us prove the proposition by induction on $m$.
%\begin{enumerate}
%\item 
For $m=1$, we have just proved the result.
%\item 
Suppose now that $\alpha_m\leq p^{-m}$ for some $m\geq 1$.
Because of the assumption above,
$\grad \psi$ is a constant vector which is nonzero modulo $p$. Applying Lemma \ref{Hensel} for $k=m$, 
we find that each zero modulo $p^m$ of $\psi$ gives rise to exactly $p^{d-1}$ zeros modulo $p^{m+1}$, which
proves that $\alpha_{m+1}\leq p^{-(m+1)}$. This concludes the induction step and the proof.
%\end{enumerate}
\end{proof}
Exploiting this corollary, we can now prove a bound on more general local densities.
\begin{prop}\label{prop:locbounds}
Let $\Psi=(\psi_1,\ldots,\psi_t)$ be a system of integral affine
linear forms in $d$ variables and $p$ be a prime so that the system reduced modulo $p$ is of finite complexity, i.e. no
two of the forms are affinely related modulo $p$. Then 
$$\alpha=\alpha_{\Psi}(p^{a_1},\ldots,p^{a_t})\leq p^{-\max_{i\neq j}(a_i+a_j)}.$$
\end{prop}
\begin{proof}If all $a_i$ are zero, the result is trivial, so 
let $m=\max a_i$ and suppose $m\geq 1$; let $i<j$ be such that $a_i+a_j$ is maximal (in
particular, it is at least $m$).
Suppose first that either $a_i$ or $a_j$ is 0. Without
loss of generality, suppose $a_i=0$ and $a_j\neq 0$.
Then for $n\in (\Z/p^m\Z)^d$ to satisfy $\psi_k(n)\equiv 0\mod p^{a_k}$ for all $k=1,\ldots,t$, we must have in particular
$\tilde{\psi_j}(\tilde{n})\equiv 0\mod p^{a_j}$,
%where $\tilde{\cdot}$ is the reduction modulo $p$,
%which imposes that $\tilde{n}$ lies in $\ker\tilde{\psi_j}$
%(our use of the notation $\ker$ may be slightly improper, because $\psi_j$ is an affine form, and not a linear one,
%on the space $\F_p^d$).
%By hypothesis, none of the reduction modulo $p$
%of the forms are 0, in particular not $\tilde{\psi_j}$. If
%its linear part is 0, then its constant part is nonzero,
%thus $\ker\tilde{\psi_j}=\emptyset$ and $\alpha_{\Psi}(p^{a_1},\ldots,p^{a_t})=0$. Otherwise, the linear part is nonzero
%modulo $p$, and then $\ker\tilde{\psi_j}$ is an affine $\F_p$-hyperplane, thus has $p^{d-1}$ elements. Moreover
%$\grad \psi_j$ is a constant vector, which is nonzero modulo $p$. Applying the previous lemma for $k=1$, and then $k=2,3,\ldots$, we find that $\psi_j$ has $p^{2d-2}$ zeros modulo
%$p^2$ and by an immediate induction, $p^{a_j(d-1)}$ zeros
%in $(\Z/p^{a_j}\Z)^d$.
and using Corollary \ref{cor:localdens}, we find that 
$$\alpha=\E_{n\in(\Z/p^m\Z)^d}\prod_{i\in[t]}1_{p^{a_i}
\mid\phi_i(n)}\leq \E_{n\in(\Z/p^{a_j}\Z)^d}1_{p^{a_j}\mid
\phi_j(n)}=p^{-a_j}=p^{-\max_{i\neq j}(a_i+a_j)}.$$
%%%%
Now suppose $1\leq a_i\leq a_j$.
Then for $n\in (\Z/p^m\Z)^d$ to satisfy $\psi_k(n)\equiv 0\mod p^{a_k}$ for all $k=1,\ldots,t$, we must have in particular
$\tilde{\psi_i}(\tilde{n})\equiv \tilde{\psi_j}(\tilde{n})\equiv 0\mod p$.
%where $\tilde{\cdot}$ is the reduction modulo $p$.
We can again suppose that $\tilde{\psi_j}$
as well as $\tilde{\psi_i}$ have linear parts which are not
0 modulo $p$, otherwise $\alpha=0$.
This imposes that $\tilde{n}$ lies in the intersection
of two affine $\F_p$-hyperplanes, namely $\ker\tilde{\psi_i}$
and $\ker\tilde{\psi_j}$ which are distinct because these forms are affinely independent modulo $p$. This intersection
is empty (and then $\alpha=0$) if and only if these hyperplanes are parallel (hence
the linear parts of $\psi_j$ and $\psi_i$ are proportional
modulo $p$). So let us suppose that the linear parts are not
proportional modulo $p$, which amounts to saying that the constant
vectors
$\grad\tilde{\psi_j},\grad\tilde{\psi_i}\in(\Z/p\Z)^d$
are not proportional.
Now we use induction on $m\geq 1$ to show that
$$
\beta_m=\Pr_{n\in\Zp{m}{d}}(\psi_i(n)\equiv \psi_j(n)\equiv 0\mod p^m)\leq p^{-2m}.
$$
%\begin{enumerate}
%\item 
For $m=1$, what we have seen above implies that $\beta_1=0$ or
$\beta_1=p^{-2}$ (the intersection of two nonparallel affine hyperplanes
of $\F_p^d$ is an affine subspace of dimension $d-2$, so its cardinality is $p^{d-2}$),
so the statement is true.
%\item 
Suppose now that for some $m\geq 1$ we have $\beta_m\leq p^{-2m}$.
If $x\in\Zp{m}{d}$ satisfies $\psi_i(x)\equiv \psi_j(x)\equiv 0\mod p^m$
 %(there are $\beta_m$ such $x$)
 and if
$y=x+p^mz\in (\Z/p^{m+1}\Z)^d$ for some $z\in(\Z/p\Z)^d$ satisfies
$\psi_i(y)\equiv\psi_j(y)\equiv 0\mod p^{m+1}$, then following
the proof of Lemma \ref{Hensel}, we infer that $z$ has to
satisfy two affine equations 
$$a+\grad\psi_i\cdot z\equiv 0\mod p\quad\text{ and }\quad a+\grad\psi_j\cdot z\equiv 0\mod p.$$
This forces $z$ to lie in the intersection of two nonparallel affine $\F_p$-hyperplanes of $\F_p^d$
(they are nonparallel because we supposed that the gradients were not proportional).
Hence for a fixed $x$ as above, there are $p^{d-2}$ such $y$,
so finally
$\beta_{m+1}=p^{d-2}\beta_m$ whence the conclusion.
%\end{enumerate}
In particular, putting $m=a_i$, we have that
$\E_{n\in\Zp{a_i}{d}}1_{\phi_i(n)\equiv\phi_j(n)\equiv 0\mod p^{a_i}}\leq p^{-2a_i}$.
It remains to induct on $a_j-a_i\geq 0$ using Lemma \ref{Hensel} in order to find that
$$\E_{n\in\Zp{a_j}{d}}1_{p^{a_i}\mid\phi_i(n)}1_{p^{a_j}\mid\phi_j(n)}\leq p^{-(a_i+a_j)},$$
which implies the desired result.
\end{proof}
We prove another statement which is helpful during the proof of the linear forms conditions (Appendix B).
\begin{prop}\label{pcarre}
Let $\Phi :\Z^d\rightarrow\Z^t$ be a system of affine-linear forms. Let $p$ be a prime such that
the reduction modulo $p$ of the system is of finite complexity.
Let $K\subset [-B,B]^d$ be a convex body. Then
$$\sum_{n\in K\cap\Z^d}1_{p^2\mid\prod_{i\in[t]}\phi_i(n)}\ll_t p^{-2}\Vol(K)+B^{d-1}p^2.$$
\end{prop}
\begin{proof}
First, we remark that $p^2\mid\prod_{i\in[t]}\phi_i(n)$ implies that either
there exists $i\in[t]$ such that $p^2\mid\phi_i(n)$ or there exist $i\neq j$ such that
$p\mid\phi_i(n)$ and $p\mid\phi_j(n)$.
Hence
$$
\sum_{n\in K\cap\Z^d}1_{p^2\mid\prod_{i\in[t]}\phi_i(n)}\leq \sum_{i\in[t]}
\sum_{n\in K\cap\Z^d}1_{p^2\mid\phi_i(n)}
+\sum_{i\neq j}\sum_{n\in K\cap\Z^d}1_{\phi_i(n)\equiv\phi_j(n)\equiv 0\mod p}.
$$
Now for any $i\in[t]$ we
apply Lemma \ref{lm:localdens} which implies
$$\sum_{n\in K\cap\Z^d}1_{p^2\mid\phi_i(n)}=\Vol (K)\alpha_{\phi_i}(p^2)+O(B^{d-1}p^2)$$
and for any $i\neq j$
$$\sum_{n\in K\cap\Z^d}1_{\phi_i(n)\equiv\phi_j(n)\equiv 0\mod p}=\Vol(K)\alpha_{\phi_i,\phi_j}(p,p)+O(B^{d-1}p).$$
But the assumption of finite complexity modulo $p$ means that we
may invoke Proposition \ref{prop:locbounds}, which implies 
that $\alpha_{\phi_i}(p^2)\leq p^{-2}$ and that $\alpha_{\phi_i,\phi_j}(p,p)\leq p^{-2}$.
The result then follows.
\end{proof}

\end{document}